\documentclass[reqno]{amsart}

\usepackage{amsmath,amsfonts,amssymb,amsthm}
\usepackage{mathtools}
\usepackage{tikz}
\usepackage{stmaryrd}
\usepackage{latexsym}
\usepackage{cite}
\usepackage{subfig}
\numberwithin{equation}{section}

\usepackage[hang,flushmargin]{footmisc}

\def\etal{\emph{et al.\/}}

\def\@Rref#1{\hbox{\rm \ref{#1}}}
\def\Rref#1{\@Rref{#1}}

\theoremstyle{plain}

\newtheorem{thrm}{Theorem}[section]
\newtheorem{lmm}[thrm]{Lemma}

\newtheorem{prpstn}[thrm]{Proposition}

\theoremstyle{definition}

\newtheorem{rmrk}[thrm]{Remark}

\theoremstyle{plain}

\newenvironment{acknowledgement}{\par\addvspace{17pt}\small\rmfamily\noindent}{\par\addvspace{6pt}}

\theoremstyle{definition}
\newtheorem{step}{Step}

\def\xCn#1{{\rm C}^#1}

\def\xLinfty{{\rm L}^{\infty}} 
\def\xLn#1{{\rm L}^#1}

\def\xdrv#1#2{\frac{{\rm d}#1}{{\rm d}#2}}

\begin{document}
	\title[Controllability of ideal incompressible magnetohydrodynamic duct flows]{Global exact controllability of ideal incompressible magnetohydrodynamic flows through a planar duct}
	
	\author[Manuel Rissel]{Manuel Rissel$^{1}$}\thanks{$^{1}$School of Mathematical Sciences, Shanghai Jiao Tong University, Shanghai, 200240, P. R. China;  manuel.rissel@sjtu.edu.cn}
	\author[Ya-Guang Wang]{Ya-Guang Wang$^2$}\thanks{$^2$School of Mathematical Sciences, Center for Applied Mathematics, MOE-LSC and SHL-MAC, Shanghai Jiao Tong University, Shanghai, 200240, P. R. China; ygwang@sjtu.edu.cn}

	\begin{abstract}
		This article is concerned with the global exact controllability for ideal incompressible magnetohydrodynamics in a rectangular domain where the controls are situated in both vertical walls. First, global exact controllability via boundary controls is established for a related Els\"asser type system by applying the return method, introduced in [Coron J.M., Math. Control Signals Systems, 5(3) (1992) 295--312]. Similar results are then inferred for the original magnetohydrodynamics system with the help of a special pressure-like corrector in the induction equation. Overall, the main difficulties stem from the nonlinear coupling between the fluid velocity and the magnetic field in combination with the aim of exactly controlling the system. In order to overcome some of the obstacles, we introduce ad-hoc constructions, such as suitable initial data extensions outside of the physical part of the domain and a certain weighted space.
		
	\end{abstract}
	\maketitle

	\section*{Introduction}
	The property of global exact controllability by means of boundary controls has been proved for perfect fluids, which are described by the incompressible Euler equations, already more than two decades ago for various kinds of bounded domains under employment of the return method introduced by Coron in \cite{returnmethod_Origin_Coron}. For instance in \cite{BoundaryControl_PerfectFluid_2D_Coron_1,BoundaryControl_PerfectFluid_2D_Coron_2} Coron obtains controllability for two dimensional simply- and multi-connected domains while the situation of three dimensional multi-connected domains is treated in \cite{BoundaryControlEuler3D_Glass} by Glass. However, in the context of ideal incompressible magnetohydrodynamics, which contains the incompressible Euler system coupled in a nonlinear way with the induction equation for the evolution of the magnetic field, similar results on controllability are not available up to the present moment. The goal of this work is to make a first step in this direction for the very particular case of a rectangular domain.
	
	Let $\Omega \subseteq \mathbb{R}^2$ be the bounded rectangular domain $\Omega := (0, L) \times (0, W)$ with boundary $\Gamma := \partial \Omega$ and outward unit normal denoted by $\mathbf{n} = (n_1, n_2)'$ as well as tangential field $\boldsymbol{\tau} := (n_2, -n_1)'$. The length $L > 0$ and width $W > 0$ are hereby fixed and the controlled part of the boundary $\Gamma_0 \subseteq \Gamma$ consists of the two vertical walls:
	\[
	\Gamma_0 := \Big( \{0\}\times(0,W) \Big) \cup \Big( \{L\}\times(0,W) \Big).
	\]

	We consider in $\Omega$ the following equations for an ideal incompressible magnetohydrodynamic flow with velocity $\mathbf{u} = (u_1, u_2)' \in \mathbb{R}^2$ and total pressure $p \in \mathbb{R}$, in the presence of a magnetic field $\mathbf{H} = (H_1, H_2)' \in \mathbb{R}^2$:
	\begin{equation}\label{equation:MHD00}
		\begin{cases}
			\partial_t \mathbf{u} + (\mathbf{u} \cdot \boldsymbol{\nabla}) \mathbf{u} - \mu(\mathbf{H} \cdot \boldsymbol{\nabla})\mathbf{H} + \boldsymbol{\nabla} p = \mathbf{0}, & \mbox{ in } \Omega \times (0, T),\\
			\partial_t \mathbf{H} + (\mathbf{u} \cdot \boldsymbol{\nabla}) \mathbf{H} - (\mathbf{H} \cdot \boldsymbol{\nabla}) \mathbf{u} = \mathbf{0}, & \mbox{ in } \Omega \times (0, T),\\
			\operatorname{div}(\mathbf{u}) = \operatorname{div}(\mathbf{H}) = 0,  & \mbox{ in } \Omega \times (0, T),\\
			\mathbf{u} \cdot \mathbf{n} = \mathbf{H} \cdot \mathbf{n} = 0,  & \mbox{ on } (\Gamma\setminus\Gamma_0) \times (0, T),\\
			\mathbf{u}(\cdot, 0)  =  \mathbf{u}_0,\, \mathbf{H}(\cdot, 0)  =  \mathbf{H}_0  & \mbox{ in } \Omega.
		\end{cases}
	\end{equation}
	Above, the vector fields $\mathbf{u}_0$ and $\mathbf{H}_0$ are the prescribed initial data at time $t = 0$ and $\mu > 0$ stands for the magnetic permeability. Moreover, the set $\Gamma_0 \subseteq \Gamma$, as specified before, represents the controlled part of the boundary through which trajectories may be influenced by applying boundary controls, and the objective is then to reach a prescribed final state $(\mathbf{u}_T, \mathbf{H}_T)$ within any small but fixed time $T > 0$, that is
	\begin{equation}\label{equation:endcondition_intro00}
		\mathbf{u}(\cdot, T)  =  \mathbf{u}_T,\, \mathbf{H}(\cdot, T)  =  \mathbf{H}_T  \mbox{ in } \Omega.
	\end{equation}
	As further discussed below in Remark~\Rref{remark:bc}, the boundary controls are obtained in an implicit way and do not explicitly appear in the system \eqref{equation:MHD00}.
	
	By employing characteristic variables, sometimes called the Els\"asser variables, defined as $\mathbf{z}^+ := \mathbf{u} + \sqrt{\mu}\mathbf{H}$ and $\mathbf{z}^- := \mathbf{u} - \sqrt{\mu}\mathbf{H}$, the above control problem \eqref{equation:MHD00} - \eqref{equation:endcondition_intro00} can be transformed into the following system:
	\begin{equation}\label{equation:MHD_Elsaesser}
		\begin{cases}
			\partial_t \mathbf{z}^+ + (\mathbf{z}^- \cdot \boldsymbol{\nabla}) \mathbf{z}^+ + \boldsymbol{\nabla} p^+ = \mathbf{0} , & \mbox{ in } \Omega \times (0, T),\\
			\partial_t \mathbf{z}^- + (\mathbf{z}^+ \cdot \boldsymbol{\nabla}) \mathbf{z}^- +\boldsymbol{\nabla} p^- = \mathbf{0}, & \mbox{ in } \Omega \times (0, T),\\
			\operatorname{div}(\mathbf{z}^+) = \operatorname{div}(\mathbf{z}^-) = 0,  & \mbox{ in } \Omega \times (0, T),\\
			\mathbf{z}^+ \cdot \mathbf{n} = \mathbf{z}^- \cdot \mathbf{n} = 0,  & \mbox{ on } (\Gamma\setminus\Gamma_0) \times (0, T),\\
			\mathbf{z}^{\pm}(\cdot, 0) = \mathbf{z}_0^{\pm} :=  \mathbf{u}_0 \pm \sqrt{\mu} \mathbf{H}_0, & \mbox{ in } \Omega,\\
			\mathbf{z}^{\pm}(\cdot, T) = \mathbf{z}_T^{\pm} :=   \mathbf{u}_T \pm \sqrt{\mu} \mathbf{H}_T, & \mbox{ in } \Omega.
		\end{cases}
	\end{equation}
	Hereby, \eqref{equation:MHD_Elsaesser} is equivalent to \eqref{equation:MHD00} - \eqref{equation:endcondition_intro00} as long as $\boldsymbol{\nabla} p^+ = \boldsymbol{\nabla} p^-$, which however would impose requirements on the controlled solutions along $\Gamma_0$. Indeed, by taking the divergence in the equation for $\mathbf{H}$ and also by multiplying with $\mathbf{n}$ along $\Gamma$, one can obtain for the difference $q := (2\sqrt{\mu})^{-1} (p^+-p^-)$ that
	\begin{equation}\label{equation:qq}
		\begin{cases}
			\Delta q = 0, & \mbox{ in } \Omega \times (0,T),\\
			\partial_{\mathbf{n}} q =  -\operatorname{sign}(n_1)\Big(\partial_t H_1 + (u_1 \partial_1 + u_2 \partial_2) H_1 - (H_1 \partial_1 + H_2 \partial_2) u_1\Big) , & \mbox{ on } \Gamma_0 \times (0,T),\\
			\partial_{\mathbf{n}} q = 0, & \mbox{ on } (\Gamma\setminus\Gamma_0) \times(0,T).
		\end{cases}
	\end{equation}
	
	In this article we shall prove the small-time global exact boundary controllability for the Els\"asser system \eqref{equation:MHD_Elsaesser} whereby the obtained solutions may have the property $\boldsymbol{\nabla} p^+ \neq \boldsymbol{\nabla} p^-$. We are not able to verify for our controlled solutions that $\boldsymbol{\nabla} p^+ = \boldsymbol{\nabla} p^-$ holds, and for the sake of drawing conclusions regarding the original system \eqref{equation:MHD00} as well, our remedy is to allow for a pressure like corrector $\boldsymbol{\nabla}q$ in the induction equation, see Theorem~\Rref{theorem:main1} below. We leave it as an open problem whether in our controllability result the harmonic contribution $q$ actually satisfies $\boldsymbol{\nabla}q \equiv 0$ or not. One possible approach for showing $\boldsymbol{\nabla}q \equiv 0$, which is under ongoing investigation, would be to establish suitable well-posedness results for \eqref{equation:MHD00} with non-characteristic boundary conditions at $\Gamma_0$.
	
	\section{Main results and outline of the work}
	For $m \in \mathbb{N}_0 := \mathbb{N}\cup\{0\}$, $n \in \mathbb{N}$, $\alpha \in (0,1)$, the H\"older spaces of $\mathbb{R}^n$-valued functions $\mathbf{f} = (f_1,\dots,f_n)'$, in the case $n = 1$ written as $f$, on a bounded domain $O \subseteq \mathbb{R}^k$ are denoted by $\left(\xCn{{m,\alpha}}(\overline{O}; \mathbb{R}^n), \|\cdot\|_{m, \alpha, O}\right)$, where
	\[
	\|\mathbf{f}\|_{m, \alpha, O} := \sum\limits_{0\leq|\boldsymbol{\boldsymbol{\beta}}|\leq m} \sup\limits_{\mathbf{x} \in O}|\partial^{\boldsymbol{\beta}}\mathbf{f}(\mathbf{x})| + \sum\limits_{|\boldsymbol{\beta}|= m} \sup\limits_{\substack{\mathbf{x},\mathbf{y} \in O \\ \mathbf{x}\neq \mathbf{y}}}\frac{|\partial^{\boldsymbol{\beta}}\mathbf{f}(\mathbf{x})-\partial^{\boldsymbol{\beta}}\mathbf{f}(\mathbf{y})|}{|\mathbf{x}-\mathbf{y}|^{\alpha}}.
	\]
	Above, $|\cdot|$ is the Euclidean norm in $\mathbb{R}^n$ and if $n=1$, the short form $\xCn{{m,\alpha}}(\overline{O})$ might be used. Furthermore, $\xLn{p}$, $\mathcal{D}'$ and $\xCn{m}$ stand for Lebesgue spaces with $p\in [1, \infty]$, the spaces of distributions and $m$-times continuously differentiable functions, $m \geq 0$, respectively. Regarding the domain $\Omega$, the following space is occasionally employed:
	\[
	\xCn{{m,\alpha}}_{\sigma,\Gamma_0}(\overline{\Omega}; \mathbb{R}^2) := \{ \mathbf{f} \in \xCn{{m,\alpha}}(\overline{\Omega}; \mathbb{R}^2) \, | \, \operatorname{div}(\mathbf{f}) = 0 \mbox{ in } \Omega, \mathbf{f} \cdot \mathbf{n} = 0 \mbox{ on } \Gamma\setminus\Gamma_0 \}.
	\]
	Moreover, concerning spatial derivatives, the following notations are mainly used:
	\begin{itemize}
		\item $\partial^{\boldsymbol{\beta}}$, with $\boldsymbol{\beta} \in \left(\mathbb{N}_0\right)^2$, for spatial derivatives in general and  $\partial_l^m$ ($\partial_l$ if $m =1$), $l = 1,2$, for the partial derivatives of order $m \in \mathbb{N}_0$ in direction ${\bf e}_l$, where $(\mathbf{e}_1,\mathbf{e}_2 )$ is the standard orthonormal basis of $\mathbb{R}^2$.
		\item $\boldsymbol{\nabla}$ and $\operatorname{div}$ denote the gradient and the divergence respectively.
		\item If $\psi$ is a differentiable scalar function of two variables and $\mathbf{v} = (v_1, v_2)'$ a differentiable vector field of two variables, then $\boldsymbol{\nabla}^{\perp}\psi := (\partial_2 \psi,- \partial_1 \psi)'$ and $\operatorname{curl}(\mathbf{v}) :=  \partial_1 v_2 - \partial_2 v_1$.
	\end{itemize}
	
	The main results proved in this article are stated in Theorems~\Rref{theorem:main} and ~\Rref{theorem:main1} below.
	\begin{thrm}\label{theorem:main}
		Let the integer $\tilde{m} \geq 3$ and the control time $T > 0$ be fixed. Then, for all initial- and final data $(\mathbf{z}^+_0, \mathbf{z}^-_0, \mathbf{z}^+_T, \mathbf{z}^-_T) \in \xCn{{\tilde{m},\alpha}}_{\sigma,\Gamma_0}(\overline{\Omega}; \mathbb{R}^2)^4$, there exists a solution $(\mathbf{z}^+, \mathbf{z}^-,p^+, p^-)$ to the control problem \eqref{equation:MHD_Elsaesser} such that
		\begin{equation}\label{equation:theorem1}
			\begin{aligned}
				(\mathbf{z}^+, \mathbf{z}^-, p^+, p^-) & \in &&  \left[ \xCn{0}([0,T];\xCn{{1,\alpha}}(\overline{\Omega}; \mathbb{R}^2)) \cap \xLinfty([0,T];\xCn{{\tilde{m},\alpha}}(\overline{\Omega}; \mathbb{R}^2)) \right]^2 \times \mathcal{D}'(\Omega\times(0,1))^2.
			\end{aligned}
		\end{equation}
	\end{thrm}
	
	Regarding the original system of incompressible ideal magnetohydrodynamics \eqref{equation:MHD00}, we conclude the following small-time global exact controllability result.
	\begin{thrm}\label{theorem:main1}
		Let the integer $\tilde{m} \geq 3$ and $T > 0$ be fixed. Then, for all initial- and final data $(\mathbf{u}_0, \mathbf{H}_0, \mathbf{u}_T, \mathbf{H}_T) \in \xCn{{\tilde{m},\alpha}}_{\sigma,\Gamma_0}(\overline{\Omega}; \mathbb{R}^2)^4$ there exists a solution
		\begin{equation*}
			\begin{aligned}
				(\mathbf{u}, \mathbf{H}, p, q) & \in && \left[ \xCn{0}([0,T];\xCn{{1,\alpha}}(\overline{\Omega}; \mathbb{R}^2)) \cap \xLinfty([0,T];\xCn{{\tilde{m},\alpha}}(\overline{\Omega}; \mathbb{R}^2)) \right]^2 \times \mathcal{D}'(\Omega\times(0,1))^2,
			\end{aligned}
		\end{equation*}
		with $q(\cdot,t)$ being for each $t \in [0,T]$ a harmonic function as in \eqref{equation:qq}, to the system
		\begin{equation}\label{equation:MHD}
			\begin{cases}
				\partial_t \mathbf{u} + (\mathbf{u} \cdot \boldsymbol{\nabla}) \mathbf{u} - \mu(\mathbf{H} \cdot \boldsymbol{\nabla})\mathbf{H} + \boldsymbol{\nabla} p = \mathbf{0}, & \mbox{ in } \Omega \times (0, T),\\
				\partial_t \mathbf{H} + (\mathbf{u} \cdot \boldsymbol{\nabla}) \mathbf{H} - (\mathbf{H} \cdot \boldsymbol{\nabla}) \mathbf{u} + \boldsymbol{\nabla} q = \mathbf{0}, & \mbox{ in } \Omega \times (0, T),\\
				\operatorname{div}(\mathbf{u}) = \operatorname{div}(\mathbf{H}) = 0,  & \mbox{ in } \Omega \times (0, T),\\
				\mathbf{u} \cdot \mathbf{n} = \mathbf{H} \cdot \mathbf{n} = 0,  & \mbox{ on } (\Gamma\setminus\Gamma_0) \times (0, T),\\
				\mathbf{u}(\cdot, 0)  =  \mathbf{u}_0,\, \mathbf{H}(\cdot, 0)  =  \mathbf{H}_0  & \mbox{ in } \Omega,
			\end{cases}
		\end{equation}
		and fulfilling the terminal condition
		\begin{equation}\label{equation:MHD-Endcondition}
			\begin{aligned}
				\mathbf{u}(\cdot, T)  =  \mathbf{u}_T,\, \mathbf{H}(\cdot, T)  =  \mathbf{H}_T  & \mbox{ in } \Omega.
			\end{aligned}
		\end{equation}
	\end{thrm}
	
	\begin{rmrk}
		In order to avoid inconveniences caused by the corners of $\Omega$, one could replace $\Omega$ for instance by a rounded rectangle. Then the controlled boundary would consist of the two lateral curves connecting the bottom and top side. For example, instead of $\Omega$ one could consider the system \eqref{equation:MHD00} directly in $\Omega_1$, as introduced below in Section~\Rref{section:prelim}, without changing the proofs.
	\end{rmrk}
	
	\begin{rmrk}\label{remark:bc}
		Instead of constructing the controls acting on $\Gamma_0$ explicitly, \eqref{equation:MHD_Elsaesser} and \eqref{equation:MHD} are treated as underdetermined problems. There are thus different possibilities for choosing boundary controls by means of suitably restricting solutions to $\Gamma_0$. This approach has been discussed for instance in \cite{BoundaryControl_PerfectFluid_2D_Coron_1,BoundaryControl_PerfectFluid_2D_Coron_2,1996Navier_Coron,BoundaryControl_NavierBC_Coron_Marbach_Sueur}. In the present case, a natural example for such a choice of controls in the Els\"asser system would be
		\begin{equation}\label{equation:control}
			\begin{cases}
				\mathbf{z}^+ \cdot \mathbf{n} \mbox{ and } \mathbf{z}^- \cdot \mathbf{n} \mbox{ on } \Gamma_0 \times (0,1) \mbox{ such that } \int_{\Gamma_0} \mathbf{z}^+ \cdot \mathbf{n} \, d\Gamma = \int_{\Gamma_0} \mathbf{z}^- \cdot \mathbf{n} \, d\Gamma = 0,\\
				\mathbf{z}^+ \mbox{ on } \{ \mathbf{x} \in \Gamma_0 \, | \, \mathbf{z}^-(\mathbf{x}) \cdot \mathbf{n}(\mathbf{x}) < 0 \},\\
				\mathbf{z}^- \mbox{ on } \{ \mathbf{x} \in \Gamma_0 \, | \, \mathbf{z}^+(\mathbf{x}) \cdot \mathbf{n}(\mathbf{x}) < 0 \}.
			\end{cases}
		\end{equation}
		Since our results include the existence of solutions, knowing uniqueness renders the given approach already meaningful. For this purpose, one can use $\xLn{2}$ energy estimates for the difference $\mathbf{Z}^{\pm} = \mathbf{z}^{\pm,1} - \mathbf{z}^{\pm,2}$ of two solutions $\mathbf{z}^{\pm,i}$, $i=1,2$ to \eqref{equation:theorem1}, having the same data:
		\begin{equation*}
			\begin{aligned}
				& \|\mathbf{Z}^{+}\|_{\xLn{2}(\Omega)}^2(t) + \|\mathbf{Z}^{-}\|_{\xLn{2}(\Omega)}^2(t) \\
				& \leq 2\max\left\{ \sup\limits_{s \in [0,T]}\|\mathbf{z}^{+,2}\|_{\xCn{1}(\overline{\Omega})}(s),  \sup\limits_{s \in [0,T]} \|\mathbf{z}^{-,2}\|_{{\xCn{1}}(\overline{\Omega})}(s)\right\} \int_0^t \left[\|\mathbf{Z}^{+}\|_{\xLn{2}(\Omega)}^2(s)+\|\mathbf{Z}^{-}\|_{\xLn{2}(\Omega)}^2(s)\right] \, ds.
			\end{aligned}
		\end{equation*}
		Due to \eqref{equation:control}, the boundary terms which appear when deriving the above formula by means of integration by parts either vanish or have a good sign.
	\end{rmrk}
	
	\begin{rmrk} The assumption $\tilde{m} \geq 3$ in Theorems~\Rref{theorem:main} and ~\Rref{theorem:main1} is employed in \eqref{equation:reg3}. But also at other points of Section~\Rref{section:proof_local_result}, for the sake of applying Banach's Fixed Point Theorem, the given data has to be of higher regularity, say $\xCn{{2,\alpha}}(\overline{\Omega}; \mathbb{R}^2)$ compared to $\xCn{0}([0,T];\xCn{{1,\alpha}}(\overline{\Omega}; \mathbb{R}^2))$ for a corresponding controlled solution, which can likewise be observed in related works such as \cite{BoundaryControl_EulerBoussinesq_Cara_Santos_Souza,BoundaryControlEuler3D_Glass}.
	\end{rmrk}
	
	\begin{rmrk}
		The choice of a rectangular domain allows us to write \eqref{equation:sourceterm_structured_y}, which is essential for using the weight $\omega_{k}$ in Section~\Rref{section:proof_local_result} (for example in the estimate \eqref{equation:auxilliary:bound_j+_b}). The crucial part in \eqref{equation:auxilliary:bound_j+_b} is the possibility of making $\omega_k(t)\int_0^1 \omega_k(s)^{-2} \, ds$ arbitrarily small by increasing $k > 0$. Hereby, the squared inverse of the weight inside of the integral is due to \eqref{equation:sourceterm_structured_y} and we could not obtain such a weight that also works for $\omega_k(t)\int_0^1 \omega_k(s)^{-1} \, ds$.
	\end{rmrk}
	
	\subsection{Related literature and outline of the proofs}
	Up to our knowledge, the question of exact controllability for incompressible ideal magnetohydrodynamics and the corresponding Els\"asser systems has not been addressed in the mathematical literature before. This is in contrast to the situation for perfect fluids with $\mathbf{H} \equiv 0$, where global exact controllability for the incompressible Euler system has been shown for instance in \cite{BoundaryControl_PerfectFluid_2D_Coron_1} by Coron, in \cite{BoundaryControlEuler3D_Glass} by Glass and with Boussinesq heat effects taken into account by Fern\'{a}ndez-Cara \etal{} in \cite{BoundaryControl_EulerBoussinesq_Cara_Santos_Souza}.
	
	In the case of magnetohydrodynamics involving viscous fluids, we can name for instance the work \cite{MHDCompr_Tao} by Tao concerning local exact controllability for compressible magnetohydrodynamics in a planar setting with controls on the whole boundary. Moreover, for incompressible fluids, Badra obtains in \cite{ContrTraj_Badra} local exact controllability to trajectories. Previous results in this direction have been achieved in \cite{LocalMHD_Barbu} by Barbu \etal{} and by Hav\^{a}rneanu \etal{} in \cite{LocalInternal2,LocalInternal1}. Moreover, Anh and Toi in \cite{ContrTraj_Anh_Toi} establish local exact controllability for a magneto-micropolar fluid model and Galan studies in \cite{ApproxMHD_Galan} approximate controllability by means of interior controls in a three-dimensional torus.
	
	Regarding basic well-posedness results for ideal incompressible magnetohydrodynamics in bounded domains, we refer to the study \cite{Schmidt} by Schmidt and to further results in this direction which have been shown by Secchi in \cite{Secchi2}. Moreover, in the article \cite{MHD_Bardos_Sulem} by Bardos \etal, long-time a priori estimates in H\"older norms for solutions in the full space case have been given.
	
	In the absence of magnetic fields, when the Navier-Stokes or Euler equations are considered, many results on controllability have been obtained so far. Recent progress, open problems and further references may for instance be found in the articles \cite{BoundaryControl_NavierBC_Coron_Marbach_Sueur, Coron_Phantom} by Coron \etal, which are concerned with global exact controllability for incompressible Navier-Stokes equations. With respect to compressible models we refer to \cite{Ervedoza_compr1,Ervedoza_compr2} by Ervedoza \etal{} and to \cite{Molina} by Molina.
	
	For proving Theorems~\Rref{theorem:main} and ~\Rref{theorem:main1} we follow in particular the approach of \cite{BoundaryControl_PerfectFluid_2D_Coron_1,BoundaryControl_EulerBoussinesq_Cara_Santos_Souza}, but face several obstructions in the details when it comes to ideal magnetohydrodynamics:
	\begin{itemize}
		\item The strong nonlinear coupling of the equations for $\mathbf{u}$ and $\mathbf{H}$ in \eqref{equation:MHD00} seems to prevent the strategy of \cite{BoundaryControl_EulerBoussinesq_Cara_Santos_Souza}, where the return method for the Euler equations is adapted to account for the additional coupling due to Boussinesq heat effects. In order to avoid loss of regularity in the fixed point iteration, we work with the system \eqref{equation:MHD_Elsaesser} instead of \eqref{equation:MHD00}.
		\item Employing the formulation \eqref{equation:MHD_Elsaesser} has the drawback that the idea of flushing the domain with zero vorticity, as previously used for instance in \cite{BoundaryControl_PerfectFluid_2D_Coron_1,BoundaryControlEuler3D_Glass,BoundaryControl_EulerBoussinesq_Cara_Santos_Souza}, is challenged by the coupling terms.
		\item How to meet the requirements imposed by the induction equation at the controlled boundary parts, see \eqref{equation:qq}, while allowing in- and outflow? This leads to the presence of $\boldsymbol{\nabla}q$ in \eqref{equation:MHD}.
	\end{itemize}
	In order to transship some of the mentioned hindrances, we work on the linearized level with suitably defined initial data in a nonphysical part of an extended domain and moreover introduce a certain weight that enables sufficient estimates, despite the presence of inhomogeneities and modified data in the linearized systems. Finding this weight heavily relies on the special geometry of the domain, which allows for a very simple return method trajectory that is compatible with the nature of the coupling terms.
	
	This article is organized as follows. In Section~\Rref{section:prelim}, a return method trajectory is introduced and some auxiliaries are presented. In Section~\Rref{section:newVars}, Theorems~\Rref{theorem:main} and~\Rref{theorem:main1} are established under the assumption that a certain local null controllability result holds. The main part of this article is then concerned with proving this local result in Section~\Rref{section:proof_local_result}.
	
	\section{Preparations and notations}\label{section:prelim}
	In this section, following the idea of \cite{BoundaryControl_PerfectFluid_2D_Coron_1}, a return method trajectory for \eqref{equation:MHD00} is introduced and several auxiliary results will be recalled. For technical reasons, given a positive constant $l > 0$, the non-physical extensions $\Omega_1, \Omega_2$ and $\Omega_3$ are, as depicted in Figure~\Rref{Figure:duct} below, defined via
	\[
	\Omega_2 := (-l, L+l) \times (-l, W+l),
	\]
	and such that $\overline{\Omega} \subseteq \overline{\Omega}_1 \subseteq \Omega_2$, $\overline{\Omega}_2 \subseteq \Omega_3$, with $\Omega_1 \subseteq \{ \mathbf{x} = (x_1,x_2)' \in \mathbb{R}^2 \, | \, 0 \leq x_2 \leq W\}$ and $\Omega_3 \subseteq \mathbb{R}^2$, both being bounded open sets with smooth boundary and outward normal vectors denoted again by $\mathbf{n}$. Moreover, let $\gamma \in \xCn{\infty}([0,1]; [0, +\infty))$ be a non-negative smooth function such that
	\begin{itemize}
		\item  $\operatorname{supp}(\gamma) \subseteq (0,1)$,
		\item $\gamma(t) = M$ for all $t \in [1/4, 3/4]$ where $M > 0$ will be specified below.
	\end{itemize}
	\begin{figure}[ht!]
		\centering
		\resizebox{0.6\textwidth}{!}{
			\begin{tikzpicture}
				\clip(-2.2,-1.1) rectangle (9.8,3.8);
				\draw[line width=0.3mm, fill=blue!5] plot[smooth cycle] (0,0) rectangle ++(8,2);
				\draw[line width=1mm] (0,-0.015) -- (0,2.015);
				\draw[line width=1mm] (8,-0.015) -- (8,2.015);
				\draw[line width=0.1mm, dashed, rounded corners=2mm] (-0.5,0) rectangle ++(8.9,2);
				\draw[line width=0.2mm, dotted] plot[smooth cycle] (-0.75,-0.75) rectangle ++(9.5,3.5);
				\draw[line width=0.2mm, dotted, rounded corners=5mm] (-1.6,-1.1) rectangle ++(11.2,4.2);
				\coordinate[label=left:$\Omega$] (A) at (4.5,3.5);
				\coordinate[label=right:$\quad$] (B) at (4,1);
				\draw[line width=0.1mm] (A) -- (B);
				\fill (B) circle (2pt);
				\coordinate[label=left:$\Omega_2$] (A2) at (1.4,3.5);
				\coordinate[label=right:$\quad$] (B2) at (0.5,2.25);
				\draw[line width=0.1mm] (A2) -- (B2);
				\fill (B2) circle (2pt);
				\coordinate[label=left:$\Omega_3$] (A22) at (6.5,3.5);
				\coordinate[label=right:$\quad$] (B22) at (7,2.9);
				\draw[line width=0.1mm] (A22) -- (B22);
				\fill (B22) circle (2pt);
				\coordinate[label=left:$\Gamma_0$] (A3) at (-0.5,3.5);
				\coordinate[label=right:$\quad$] (B3) at (0,0.75);
				\draw[line width=0.1mm] (A3) -- (B3);
				\coordinate[label=left:$\Gamma_0$] (A4) at (8.6,3.5);
				\coordinate[label=right:$\quad$] (B4) at (8,0.75);
				\draw[line width=0.1mm] (A4) -- (B4);
				\coordinate[label=left:$\Omega_1$] (A5) at (-1.5,3.5);
				\coordinate[label=right:$\quad$] (B5) at (-0.2,0.75);
				\draw[line width=0.1mm] (A5) -- (B5);
				\fill (B5) circle (2pt);
				\draw[thin,->] (1.2,0.4) -- (2,0.4);
				\draw[thin,->] (1.2,0.8) -- (2,0.8);
				\draw[thin,->] (1.2,1.2) -- (2,1.2);
				\draw[thin,->] (1.2,1.6) -- (2,1.6);
				\draw[thin,->] (6.0,0.4) -- (6.8,0.4);
				\draw[thin,->] (6.0,0.8) -- (6.8,0.8);
				\draw[thin,->] (6.0,1.2) -- (6.8,1.2);0
				\draw[thin,->] (6.0,1.6) -- (6.8,1.6);
			\end{tikzpicture}
		}
		\caption{A sketch of the rectangular domain $\Omega$ and its extensions, while the arrows indicate $\overline{\mathbf{y}}(\mathbf{x},t)$.}
		\label{Figure:duct}
	\end{figure}
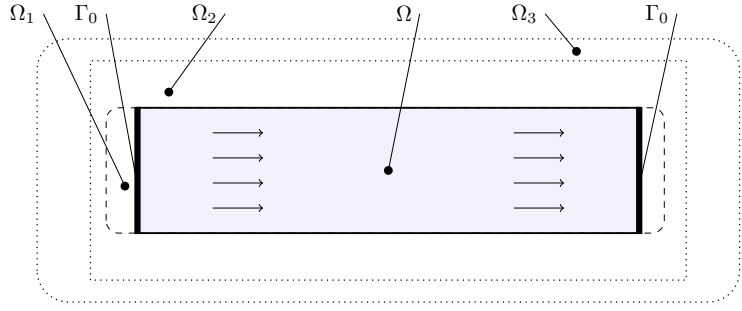
	Then, with $\chi \in \xCn{\infty}_0(\overline{\Omega_3})$ denoting a smooth cutoff function satisfying $\chi(\mathbf{x}) = 1$ for $\mathbf{x} \in \overline{\Omega}_2$, we define
	\[
	\mathbf{y}^*(\mathbf{x}, t) := \gamma(t) \chi(\mathbf{x}) \mbox{\bf e}_1 = \binom{\gamma(t) \chi(\mathbf{x})}{0}, \quad (\mathbf{x},t) \in \overline{\Omega}_3 \times [0,1].
	\]
	and in particular, the functions $(\overline{\mathbf{y}}, \overline{\mathbf{H}}, \overline{p}, \overline{q})$ which are given by
	\[
	\begin{cases}
		\overline{\mathbf{y}}(\mathbf{x}, t) := \mathbf{y}^*(x,t), & (\mathbf{x},t) \in \Omega \times [0,1],\\
		\overline{\mathbf{H}}(\mathbf{x}, t) := \mathbf{0},& (\mathbf{x},t) \in \Omega \times [0,1],\\
		\overline{p}(\mathbf{x},t) := -x_1\xdrv{}{t} \gamma(t), & (\mathbf{x},t) \in \Omega \times [0,1],\\
		\overline{q}(\mathbf{x},t) := 0, & (\mathbf{x},t) \in \Omega \times [0,1],
	\end{cases}
	\]
	where $\mathbf{x}=(x_1,x_2)'$, solve \eqref{equation:MHD} with $T = 1$ for the data $\mathbf{u}_0 = \mathbf{u}_T = \mathbf{H}_0 = \mathbf{H}_T = \mathbf{0}$.
	
	Henceforth, let $\boldsymbol{\mathcal{Y}}(\mathbf{x},s,t)$ denote for $(\mathbf{x},s,t) \in \overline{\Omega}_3\times[0,1]\times[0,1]$ the flow corresponding to $\mathbf{y}^*$ which solves the ordinary differential equation
	\begin{equation}\label{equation:flowofy}
		\begin{cases}
			\xdrv{}{t} \boldsymbol{\mathcal{Y}}(\mathbf{x},s,t) = \mathbf{y}^*(\boldsymbol{\mathcal{Y}}(\mathbf{x},s,t),t),\\
			\boldsymbol{\mathcal{Y}}(\mathbf{x},s,s) = \mathbf{x}.
		\end{cases}
	\end{equation}
	That \eqref{equation:flowofy} admits a unique solution is due to the global Lipschitz condition satisfied by $\mathbf{y}^*$. Notably, $\boldsymbol{\mathcal{Y}}$ makes every point in $\overline{\Omega}_2$ leave this domain until $t = 1$ as long as $M > 0$ is sufficiently large, which is a result that in greater generality may be found in \cite{BoundaryControl_PerfectFluid_2D_Coron_1,BoundaryControlEuler3D_Glass}.
	\begin{lmm}\label{lemma:Y}
		The constant $M > 0$ can be chosen large enough, such that $\boldsymbol{\mathcal{Y}}(\mathbf{x}, 0, 1) \notin \overline{\Omega}_2$ for all $\mathbf{x} \in \overline{\Omega}_2$.
	\end{lmm}
	\begin{proof}
		Temporarily denote $\boldsymbol{\mathcal{Y}}(\mathbf{x},t) = \boldsymbol{\mathcal{Y}}(\mathbf{x},0,t)$ and let $\theta$ be the projection $\mathbb{R}^2 \ni (x_1, x_2) \mapsto x_1$. Because of
		\[
		\partial_t (\theta(\boldsymbol{\mathcal{Y}})) (\mathbf{x},t)= \gamma(t) \chi (\boldsymbol{\mathcal{Y}}(\mathbf{x},t)),
		\]
		combined with $\chi = 1$ in $\overline{\Omega}_2$ and $\gamma(t) = M$ for $t \in [1/4, 3/4]$, one can find $M > 0$ large enough in order to guarantee
		$\theta(\boldsymbol{\mathcal{Y}}(\mathbf{x}, 1)) > L+l$ for all $\mathbf{x} \in \overline{\Omega}_2$. However, by construction one has $\theta(\overline{\Omega}_2) = [-l,L+l]$, which implies for all $\mathbf{x} \in \overline{\Omega}_2$ that $\boldsymbol{\mathcal{Y}}(\mathbf{x},1) \notin \overline{\Omega}_2$.
	\end{proof}
	
	As in \cite[Lemma 6]{BoundaryControl_EulerBoussinesq_Cara_Santos_Souza}, the following extension operators are introduced.
	\begin{lmm}\label{lemma:extensionOperators}
		There are linear continuous extension operators $\pi_i\colon \xCn{0}(\overline{\Omega};\mathbb{R}^i) \to \xCn{0}(\overline{\Omega}_3;\mathbb{R}^i)$, $i = 1,2$, such that $\pi_i(f)_{|_{\Omega}} = f$ and $\operatorname{supp}(\pi_i(f)) \subseteq \Omega_2$ for every $f \in \xCn{0}(\overline{\Omega};\mathbb{R}^i)$ and such that $\pi_i$ maps $\xCn{{m, \beta}}(\overline{\Omega};\mathbb{R}^i)$ continuously into $\xCn{{m, \beta}}(\overline{\Omega}_3;\mathbb{R}^i)$ for all integers $m \geq 0$ and $\beta \in (0,1)$.
	\end{lmm}
	A detailed proof for the following result can be found in \cite[Lemma 7]{BoundaryControl_EulerBoussinesq_Cara_Santos_Souza}.
	\begin{lmm}\label{lemma:nu}
		For $\mathbf{z} \in \xCn{0}(\overline{\Omega} \times [0,1]; \mathbb{R}^2)$ denote by $\boldsymbol{\mathfrak{z}}$ the extension $\boldsymbol{\mathfrak{z}} := \mathbf{y}^* + \pi_2(\mathbf{z} - \overline{\mathbf{y}})$ and by $\boldsymbol{\mathcal{Z}}(\mathbf{x},s,t)$ for $(\mathbf{x},s,t) \in \overline{\Omega}_3\times[0,1]\times[0,1]$ the corresponding flow, which solves
		\[
		\begin{cases}
			\xdrv{}{t} \boldsymbol{\mathcal{Z}}(\mathbf{x},s,t) = \boldsymbol{\mathfrak{z}}(\boldsymbol{\mathcal{Z}}(\mathbf{x},s,t),t),\\
			\boldsymbol{\mathcal{Z}}(\mathbf{x},s,s) = \mathbf{x}.
		\end{cases}
		\]	
		Then, there exists a small $\nu > 0$, such that $\|\mathbf{z} - \overline{\mathbf{y}}\|_{\xCn{0}(\overline{\Omega} \times [0,1])} < \nu$ implies $\boldsymbol{\mathcal{Z}}(\mathbf{x}, 0, 1) \notin \overline{\Omega}_2$ for all $\mathbf{x} \in \overline{\Omega}_2$.
	\end{lmm}
	
	The next lemma can be found in \cite[Lemma 1]{Bardos_EulerEq}, see also \cite[Lemmas 1-2]{BoundaryControl_EulerBoussinesq_Cara_Santos_Souza}, and will be employed mainly in form of Remark~\Rref{remark:Transport_Estimate} below.
	\begin{lmm}\label{lemma:Transport_Estimate}
		For $T > 0$, $m \in \mathbb{N}_0$ consider functions $\mathbf{z} \in \xCn{0}([0, T]; \xCn{{m, \alpha}}(\overline{\Omega}_3; \mathbb{R}^2))$, $v \in \xCn{0}([0, T]; \xCn{{m+1, \alpha}}(\overline{\Omega}_3))$ and $g \in \xCn{0}([0, T]; \xCn{{m, \alpha}}(\overline{\Omega}_3))$ with the properties
		\[
		\begin{cases}
			\mathbf{z} \cdot \mathbf{n} = 0, & \mbox{\normalfont on } \partial \Omega_3 \times [0,T],\\
			\partial_t v + (z \cdot \boldsymbol{\nabla}) v = g, & \mbox{\normalfont in } \Omega_3 \times (0, T).
		\end{cases}
		\]
		Then $\partial_t v \in \xCn{0}([0, T]; \xCn{{m, \alpha}}(\overline{\Omega}_3))$. Moreover, with a constant $K = K(\alpha, k) > 0$ and $\xdrv{}{t^+}$ denoting the right derivative in time, it holds for $t \in (0,T)$ that
		\[
		\xdrv{}{t^+} \|v(\cdot, t)\|_{m, \alpha,\Omega_3} \leq \begin{cases}
			\|g(\cdot, t)\|_{0, \alpha,\Omega_3} + \alpha \|\boldsymbol{\nabla} \mathbf{z}(\cdot, t)\|_{0, \alpha,\Omega_3}\|v(\cdot, t)\|_{0, \alpha,\Omega_3}, & m = 0, \\
			\|g(\cdot, t)\|_{m, \alpha,\Omega_3} + K(\alpha, k) \|\mathbf{z}(\cdot, t)\|_{m, \alpha,\Omega_3}\|v(\cdot, t)\|_{m, \alpha,\Omega_3}, & m \geq 1.
		\end{cases}
		\]
	\end{lmm}
	\begin{rmrk}\label{remark:Transport_Estimate}
		In the context of Lemma~\Rref{lemma:Transport_Estimate}, for all $t \in (0, T)$ one has the following estimate:
		\[
		\|v(\cdot, t)\|_{m, \alpha,\Omega_3} \leq \begin{cases}
			\left( \int_0^t \|g(\cdot, s)\|_{0, \alpha,\Omega_3} \, ds + \|v(\cdot, 0)\|_{0,\alpha,\Omega_3} \right) \operatorname{e}^{\alpha \int_0^t \|\mathbf{z}(\cdot, s)\|_{1, \alpha,\Omega_3} \,ds}, & m = 0, \\
			\left( \int_0^t \|g(\cdot, s)\|_{m, \alpha,\Omega_3} \, ds + \|v(\cdot, 0)\|_{m,\alpha,\Omega_3} \right) \operatorname{e}^{ K \int_0^t \|\mathbf{z}(\cdot, s)\|_{m, \alpha,\Omega_3} \,ds}, & m \geq 1.
		\end{cases}
		\]
		Hereby, it is enough to assume that $v \in \xCn{0}([0, T]; \xCn{{m, \alpha}}(\overline{\Omega}_3))$ due to regularization.
	\end{rmrk}

	\section{Proofs of the main results}\label{section:newVars}
	In this section, Theorems~\Rref{theorem:main} and~\Rref{theorem:main1} are concluded based on the assumption that a local controllability result, namely Proposition~\Rref{proposition:smalldata} as stated below, is true. The verification of Proposition~\Rref{proposition:smalldata} is postponed at this point and shall be carried out later out in Section~\Rref{section:proof_local_result}.
	
	By choosing $T = 1$ and $\mathbf{z}_T^{\pm} = \mathbf{0}$ one obtains, through applying the $\operatorname{curl}$ operator in \eqref{equation:MHD_Elsaesser}, the control problem
	\begin{equation}\label{equation:MHD_curled_Elsaesser}
		\begin{cases}
			\partial_t j^+ + (\mathbf{z}^- \cdot \boldsymbol{\nabla}) j^+ = g^+, & \mbox{ in } \Omega \times (0, 1),\\
			\partial_t j^- + (\mathbf{z}^+ \cdot \boldsymbol{\nabla}) j^- = g^-, & \mbox{ in } \Omega \times (0, 1),\\
			\operatorname{curl}(\mathbf{z}^+) = j^+, \operatorname{curl}(\mathbf{z}^-) = j^- & \mbox{ in } \Omega \times (0, 1),\\
			\operatorname{div}(\mathbf{z}^+) = \operatorname{div}(\mathbf{z}^-) = 0,  & \mbox{ in } \Omega \times (0, 1),\\
			\mathbf{z}^+ \cdot \mathbf{n} =  \mathbf{z}^- \cdot \mathbf{n} = 0,  & \mbox{ on } (\Gamma\setminus\Gamma_0) \times (0, 1),\\
			\mathbf{z}^+(\cdot,0) = \mathbf{z}^+_0, \, \mathbf{z}^-(\cdot,0) = \mathbf{z}^-_0, & \mbox{ in } \Omega,\\
			\mathbf{z}^+(\cdot,1) = \mathbf{0}, \, \mathbf{z}^-(\cdot,1) = \mathbf{0}, & \mbox{ in } \Omega,\\
			j^+(\cdot, 0) = \operatorname{curl}(\mathbf{z}_0^+),\, j^-(\cdot, 0)  =  \operatorname{curl}(\mathbf{z}_0^-)  & \mbox{ in } \Omega,\\
			j^+(\cdot, 1) = 0,\, j^-(\cdot, 1)  =  0  & \mbox{ in } \Omega.\\
		\end{cases}
	\end{equation}
	Denoting the components of $\mathbf{z}^{\pm}$ as $\mathbf{z}^{\pm} = (z^{\pm}_1, z^{\pm}_2)'$, the coupling terms $g^{\pm}$ are given by
	\begin{equation}\label{equation:sourceterm_unstructured}
		\begin{aligned}
			g^{\pm} & := && \partial_2 z^{\mp}_1 \partial_1 z^{\pm}_1 + \partial_2 z^{\mp}_2 \partial_2 z^{\pm}_1  - \partial_1 z^{\mp}_1 \partial_1 z^{\pm}_2 - \partial_1 z^{\mp}_2 \partial_2 z^{\pm}_2.
		\end{aligned}
	\end{equation}
	Utilizing the assumption that both of $\mathbf{z}^+$ and $\mathbf{z}^-$ are divergence free, the functions $g^{\pm}$ may in all of $\overline{\Omega} \times [0,1]$ be represented in the following way:
	\begin{equation}\label{equation:sourceterm_structured}
		\begin{aligned}
			g^{\pm} & = && -(\partial_1z^{\mp}_1 - \partial_1z^{\pm}_1)(\partial_1z^{\pm}_2 + \partial_2z^{\pm}_1) \\
			& &&  -(\partial_1z^{\pm}_2 - \partial_1z^{\mp}_2)\partial_1z^{\pm}_1 - (\partial_2z^{\pm}_1 - \partial_2z^{\mp}_1)\partial_1z^{\pm}_1.
		\end{aligned}
	\end{equation}
	Since $\overline{\mathbf{y}} = (\overline{y}_1, \overline{y}_2)$ is constant in all of $\overline{\Omega}$ with respect to the spatial variables, we can rewrite \eqref{equation:sourceterm_structured} as
	\begin{equation}\label{equation:sourceterm_structured_y}
		\begin{aligned}
			g^{\pm} & = && -(\partial_1z^{\mp}_1 - \partial_1z^{\pm}_1)\partial_1(z^{\pm}_2 - \overline{y}_2) -(\partial_1z^{\mp}_1 - \partial_1z^{\pm}_1)\partial_2(z^{\pm}_1 - \overline{y}_1) \\
			& &&  -(\partial_1z^{\pm}_2 - \partial_1z^{\mp}_2)\partial_1(z^{\pm}_1 - \overline{y}_1) - (\partial_2z^{\pm}_1 - \partial_2z^{\mp}_1)\partial_1(z^{\pm}_1 - \overline{y}_1).
		\end{aligned}
	\end{equation}
	
	\begin{rmrk}
		Writing the coupling terms in the form \eqref{equation:sourceterm_structured_y} will allow us in Section~\Rref{section:proof_local_result} to employ the weight $\omega_{k}$ defined in \eqref{equation:weight}. Instead of introducing \eqref{equation:sourceterm_unstructured} and then transforming via \eqref{equation:sourceterm_structured} to \eqref{equation:sourceterm_structured_y} it would be shorter, and sufficient for the proofs later on, to write directly
		\begin{equation}\label{equation:termterm_g_pm}
			g^{\pm} = -\sum\limits_{i = 1}^2 \boldsymbol{\nabla} (z^-_i-\overline{y}_i) \wedge \partial_i (\mathbf{z}^+ - \overline{\mathbf{y}}).
		\end{equation}
		However, \eqref{equation:sourceterm_structured} is valid for more general domains where the return method trajectory would be less simple and displays a structure that might be useful for the general case. Since, comparing with \eqref{equation:termterm_g_pm}, the expression \eqref{equation:sourceterm_structured_y} is closer to \eqref{equation:sourceterm_structured}, we choose to employ \eqref{equation:sourceterm_structured_y} from now on.
	\end{rmrk}

	The following local null controllability result, established later in Section~\Rref{section:proof_local_result}, constitutes the main step for proving Theorems~\Rref{theorem:main} and ~\Rref{theorem:main1}.
	\begin{prpstn}\label{proposition:smalldata}
		Let $\tilde{m} \geq 3$ be fixed. There exists a constant $\tilde{s} > 0$, such that if the initial data $(\mathbf{z}^+_0, \mathbf{z}^-_0) \in \xCn{{\tilde{m},\alpha}}_{\sigma,\Gamma_0}(\overline{\Omega}; \mathbb{R}^2)^2$ satisfy the constraint $\max\{\|\mathbf{z}^+_0\|_{\tilde{m},\alpha,\Omega},\|\mathbf{z}^-_0\|_{\tilde{m},\alpha,\Omega}\} < \tilde{s}$, then the system \eqref{equation:MHD_curled_Elsaesser} admits a solution $(\mathbf{z}^+, \mathbf{z}^-, j^+, j^-)$ of regularity
		\begin{equation*}
			\begin{aligned}
				(\mathbf{z}^+, \mathbf{z}^-, j^+, j^-) & \in &&  \left[ \xCn{0}([0,1];\xCn{{1,\alpha}}(\overline{\Omega}; \mathbb{R}^2)) \cap \xLinfty([0,1];\xCn{{\tilde{m},\alpha}}(\overline{\Omega}; \mathbb{R}^2)) \right]^2 \\
				& && \times \left[ \xCn{0}([0,1];\xCn{{0,\alpha}}(\overline{\Omega})) \cap \xLinfty([0,1];\xCn{{\tilde{m}-1,\alpha}}(\overline{\Omega})) \right]^2,
			\end{aligned}
		\end{equation*}
		with $\mathbf{z}^+(\mathbf{x}, 1) = \mathbf{z}^-(\mathbf{x}, 1) = \mathbf{0}$ for all $\mathbf{x} \in \Omega$.
	\end{prpstn}
	By construction, there are $p^+, p^- \in \mathcal{D}'(\Omega\times(0,1))$ such that $(\mathbf{z}^+, \mathbf{z}^-)$ obtained in Proposition~\Rref{proposition:smalldata} together with $(p^+, p^-)$ satisfy \eqref{equation:MHD_Elsaesser} for the case of zero final data and the final time being fixed to $T = 1$.

	Assuming first that Proposition~\Rref{proposition:smalldata} is true, we show now how to deduce Theorems~\Rref{theorem:main} and ~\Rref{theorem:main1} with the help of a scaling and gluing argument, as for instance in \cite{BoundaryControl_PerfectFluid_2D_Coron_1,BoundaryControl_EulerBoussinesq_Cara_Santos_Souza,BoundaryControlEuler3D_Glass}.
	\begin{proof}[Proof of Theorems~\Rref{theorem:main} and ~\Rref{theorem:main1}] We only verify Theorem~\Rref{theorem:main1}, since the arguments for Theorem~\Rref{theorem:main} are along the same line. We note that if $(\mathbf{u}, \mathbf{H}, p, q)$ solve \eqref{equation:MHD}, then this is also true for $(\hat{\mathbf{u}}, \hat{\mathbf{H}}, \hat{p}, \hat{q})$ defined by
		\begin{equation}\label{equation:reverse}
			\begin{aligned}
				\hat{\mathbf{\mathbf{u}}}(\mathbf{x},t) & := && - \mathbf{\mathbf{u}}(\mathbf{x}, T-t),\\
				\hat{\mathbf{H}}(\mathbf{x},t) & := && - \mathbf{H}(\mathbf{x}, T-t),\\
				\hat{p}(\mathbf{x},t) & := && p(\mathbf{x}, T-t),\\
				\hat{q}(\mathbf{x},t) & := && q(\mathbf{x}, T-t).
			\end{aligned}
		\end{equation}
		Next, the time interval $[0, T]$ is split into two parts $[0, T/2]$,  $[T/2, 1]$ and a number $0 < \epsilon < T/2$ is chosen so small that the modified initial data $\tilde{\mathbf{u}}_0 := \epsilon \mathbf{u}_0$, $\tilde{\mathbf{H}}_0 := \epsilon \mathbf{H}_0$ and final data $\tilde{\mathbf{u}}_T := \epsilon \mathbf{u}_T$, $\tilde{\mathbf{H}}_T := \epsilon \mathbf{H}_T$ satisfy
		\begin{equation*}
			\begin{aligned}
				\max\left\{\|\tilde{\mathbf{u}}_0  + \sqrt{\mu} \tilde{\mathbf{H}}_0\|_{\tilde{m},\alpha,\Omega}, \, \|\tilde{\mathbf{u}}_0  - \sqrt{\mu} \tilde{\mathbf{H}}_0\|_{\tilde{m},\alpha,\Omega}, \, \|\tilde{\mathbf{u}}_T  + \sqrt{\mu} \tilde{\mathbf{H}}_T\|_{\tilde{m},\alpha,\Omega}, \, \|\tilde{\mathbf{u}}_T  - \sqrt{\mu} \tilde{\mathbf{H}}_T\|_{\tilde{m},\alpha,\Omega}\right\} < \tilde{s},
			\end{aligned}
		\end{equation*}
		where $\tilde{s} > 0$ is the constant introduced in Proposition~\Rref{proposition:smalldata}. This leads, by applying Proposition~\Rref{proposition:smalldata} with $T = 1$ to solutions $(\mathbf{u}^{*}, \mathbf{H}^{*}, p^{*}, q^{*})$ and $(\mathbf{u}^{**}, \mathbf{H}^{**}, p^{**}, q^{**})$ to \eqref{equation:MHD}, obeying
		\[
		\begin{cases}
			(\mathbf{u}^{*}, \mathbf{H}^{*})(\cdot, 0) & = (\mathbf{u}_0(\cdot), \mathbf{H}_0(\cdot)),\\
			(\mathbf{u}^{*}, \mathbf{H}^{*}, p^{*}, q^{*})(\cdot, 1) & \equiv (\mathbf{0},\mathbf{0},0,0),\\
			(\mathbf{u}^{**}, \mathbf{H}^{**})(\cdot, 0) & = -(\mathbf{u}_T(\cdot), H_T(\cdot)),\\
			(\mathbf{u}^{**}, \mathbf{H}^{**}, p^{**}, q^{**})(\cdot, 1) & \equiv (\mathbf{0},\mathbf{0},0,0).
		\end{cases}
		\]
		Here, the functions $p^{*}(\cdot, 1), q^{*}(\cdot, 1), p^{**}(\cdot, 1)$ and $q^{**}(\cdot, 1)$ might be nonzero constants at first, but we modify them to be zero as well.
		In order to utilize the reversibility in time, which was outlined in \eqref{equation:reverse} above, we define
		\begin{equation*}
			\begin{cases}
				\left(\mathbf{u}^{a}, \mathbf{H}^{a}, p^{a}, q^{a}\right) (\mathbf{x},t) := \left(\epsilon^{-1}\mathbf{u}^{*}, \epsilon^{-1}\mathbf{H}^{*}, \epsilon^{-2}p^{*}, \epsilon^{-2}q^{*} \right) (\mathbf{x}, \epsilon^{-1}t), & (\mathbf{x},t) \in \Omega \times [0, \epsilon],\\
				\left(\mathbf{u}^{a}, \mathbf{H}^{a}, p^{a}, q^{a}\right) (\mathbf{x},t) := \left(\mathbf{0}, \mathbf{0}, 0, 0 \right), & (\mathbf{x},t) \in \Omega \times [\epsilon, T/2],
			\end{cases}
		\end{equation*}
		as well as
		\begin{equation*}
			\begin{cases}
				\left(\mathbf{u}^{b}, \mathbf{H}^{b}\right) (\mathbf{x},t) := -\left(\epsilon^{-1}\mathbf{u}^{**}, \epsilon^{-1}\mathbf{H}^{**}\right)(\mathbf{x}, \epsilon^{-1}(T-t)), & (\mathbf{x},t) \in \Omega \times [T-\epsilon, T],\\
				\left(p^{b}, q^{b}\right) (\mathbf{x},t) := \left(\epsilon^{-2}p^{**}, \epsilon^{-2}q^{**}\right)(\mathbf{x}, \epsilon^{-1}(T-t)), & (\mathbf{x},t) \in \Omega \times [T-\epsilon, T],\\
				\left(\mathbf{u}^{b}, \mathbf{H}^{b}, p^{b}, q^{b}\right) (\mathbf{x},t) := \left( \mathbf{0}, \mathbf{0}, 0, 0 \right), & (\mathbf{x},t) \in \Omega \times [T/2, T-\epsilon],
			\end{cases}
		\end{equation*}
		and observe that the functions
		\begin{equation*}
			\begin{aligned}
				\mathbf{u}(\mathbf{x},t) & := && \begin{cases}
					\mathbf{u}^{a}(\mathbf{x},t), & (\mathbf{x},t) \in \Omega \times [0, T/2],\\
					\mathbf{u}^{b}(\mathbf{x},t), & (\mathbf{x},t) \in \Omega \times [T/2, T],
				\end{cases}  \\
				\mathbf{H}(\mathbf{x},t) & := && \begin{cases}
					\mathbf{H}^{a}(\mathbf{x},t), & (\mathbf{x},t) \in \Omega \times [0, T/2],\\
					\mathbf{H}^{b}(\mathbf{x},t), & (\mathbf{x},t) \in \Omega \times [T/2, T],
				\end{cases}  \\
				p(\mathbf{x},t) & := && \begin{cases}
					p^{a}(\mathbf{x},t), & (\mathbf{x},t) \in \Omega \times [0, T/2],\\
					p^{b}(\mathbf{x},t), & (\mathbf{x},t) \in \Omega \times [T/2, T],
				\end{cases} \\
				q(\mathbf{x},t) & := && \begin{cases}
					q^{a}(\mathbf{x},t), & (\mathbf{x},t) \in \Omega \times [0, T/2],\\
					q^{b}(\mathbf{x},t), & (\mathbf{x},t) \in \Omega \times [T/2, T],
				\end{cases}
			\end{aligned}
		\end{equation*}	
		solve the control problem \eqref{equation:MHD} - \eqref{equation:MHD-Endcondition}.
	\end{proof}
	
	\section{Proof of the local null controllability result}\label{section:proof_local_result}
	This section is devoted to proving Proposition~\Rref{proposition:smalldata}. The first step is to define a certain map $F$ on a suitable subset of a Banach space such that a fixed point of this map will be a solution to \eqref{equation:MHD_curled_Elsaesser}. The next task is then to assure that this map is a well defined self map and contractive. The idea for using this kind of fixed point argument is similar to \cite{BoundaryControl_PerfectFluid_2D_Coron_1,BoundaryControlEuler3D_Glass,BoundaryControl_EulerBoussinesq_Cara_Santos_Souza}, however most parts are carried out in a different way due to the peculiarities of ideal magnetohydrodynamics.
	
	\subsection{Definition of a fixed point space}
	In what follows, $\lambda \in \xCn{\infty}([0,1]; [0,1])$ is such that $\lambda(t) = 1$ for $t \in [0, d]$ and $\lambda(t) = 0$ for $t \in [2d, 1]$ with arbitrary but fixed $d \in (0,1/2)$. Let then the two spaces $X^+$ and $X^-$ be given by
	\begin{equation*}
		\begin{aligned}
			X^{\pm}  :=  && \left\{ \mathbf{z} \in \xCn{0}([0,1];\xCn{{\tilde{m},\alpha}}(\overline{\Omega}; \mathbb{R}^2)) \, \Big| \, \operatorname{div}(\mathbf{z}) = 0 \mbox{ in } \Omega\times (0,1), \mathbf{z}(\cdot,0) = \mathbf{z}^{\pm}_0 \mbox{ in } \Omega, \mathbf{z} \cdot \mathbf{n} = 0\mbox{ on } \Gamma\setminus\Gamma_0 \times (0,1)  \right\},
		\end{aligned}
	\end{equation*}
	and define for $k > 0$ large, which will be fixed later, the weight function
	\begin{equation}\label{equation:weight}
		\omega_k(t) := \left( \frac{1}{2} + \frac{t}{8} \right)^{-k}, \quad (t \in [0,1]).
	\end{equation}
	
	Next, we introduce a set $X_{\nu, k}$ on which we shall construct a fixed point iteration below,
	\begin{equation*}
		\begin{aligned}
			X_{\nu, k} := \left\{ (\mathbf{z}^+, \mathbf{z}^-) \in X^+ \times X^- \, \Big| \, \max\limits_{t \in [0,1]} \omega_k(t) \|\mathbf{z}^{\pm} - \overline{\mathbf{y}}\|_{\tilde{m},\alpha,\Omega}(t) < \nu \mbox{ and } \mathbf{z}^{\pm}(\cdot,1) \equiv 0 \mbox{ in } \Omega \right\}.
		\end{aligned}
	\end{equation*}
	Above, $\nu > 0$ is the small constant from Lemma~\Rref{lemma:nu} and $X_{\nu, k}$ is not empty as long as $\|\mathbf{z}^+_0\|_{\tilde{m}, \alpha,\Omega}$ and $\|\mathbf{z}^-_0\|_{\tilde{m}, \alpha,\Omega}$ are, depending on $k$, chosen to be small enough. Indeed, if this is the case, one has
	\[
	(\overline{\mathbf{y}} + \lambda \mathbf{z}^+_0, \overline{\mathbf{y}} + \lambda \mathbf{z}^-_0) \in X_{\nu, k}.
	\]
	The definition of $X_{\nu, k}$ above is motivated mainly by the following requirements:
	\begin{itemize}
		\item In view of the return method trajectory $(\overline{\mathbf{y}}, \overline{\mathbf{H}} \equiv \mathbf{0}, \overline{p}, \overline{q} \equiv 0)$ one should have for each $(\mathbf{z}^+, \mathbf{z}^-) \in  X_{\nu, k}$ an estimate of the form
		\[
		\max\limits_{t \in [0,1]} \omega_k(t) \|\mathbf{z}^+ - \mathbf{z}^-\|_{\tilde{m},\alpha,\Omega}(t) \leq \max\limits_{t \in [0,1]} \omega_k(t) \|\mathbf{z}^+ - \overline{\mathbf{y}} + \overline{\mathbf{y}} - \mathbf{z}^-\|_{\tilde{m},\alpha,\Omega}(t) < 2\nu.
		\]
		\item Elements of $X_{\nu, k}$ should be near enough to $\overline{\mathbf{y}}$ such that Lemma~\Rref{lemma:nu} can be applied.
		\item The weight should guarantee that for the map $F$ defined below a fixed point can be found, but it has to be compatible with the previous requirement and satisfy for instance $\omega_k(t) \geq 1$ for all $t \in [0,1]$. In particular, $\omega_k$ as defined above obeys $1 < \omega_k(t) < +\infty$ for each choice of $t \in [0,1], k > 0$ and for every fixed $t \in [0,1]$ one has as $k \to +\infty$ that
		\begin{equation*}
			\begin{aligned}
				\omega_k(t) \int_0^t \frac{1}{\omega_k(s)} \, ds  & = && \frac{8}{k+1} \left[ \frac{\omega_k(t)}{\omega_{k+1}(t)} - \frac{\omega_k(t)}{\omega_{k+1}(0)} \right] \leq \frac{8}{k+1} \left( \frac{1}{2} + \frac{t}{8} \right)^{-k}  \left( \frac{1}{2} + \frac{t}{8} \right)^{k+1} \to 0.
			\end{aligned}
		\end{equation*}
		Moreover, for all $t \in [0,1]$ it holds that
		\begin{equation}\label{equation:weight_trick}
			\begin{aligned}
				\omega_k(t) \int_0^1 \frac{1}{\omega_k(s)^2} \, ds  & = && \omega_k(t)\frac{8}{2k+1} \left[ \frac{1}{\omega_{2k+1}(t)} - \frac{1}{\omega_{2k+1}(0)} \right] \\
				& \leq && \frac{8}{2k+1} \left( \frac{1}{2} + \frac{t}{8} \right)^{-k}  \left( \frac{1}{2} + \frac{1}{8} \right)^{2k}  \left( \frac{1}{2} + \frac{1}{8} \right)\\
				& \leq && \frac{5}{2k+1} 2^{k}  \left( \frac{25}{64} \right)^{k} \\
				& \leq && \frac{5}{2k+1} \to 0, \, \mbox{ as } k \to +\infty.
			\end{aligned}
		\end{equation}
		In order to make use of \eqref{equation:weight_trick} later on, the representation \eqref{equation:sourceterm_structured_y} for $g^{\pm}$ in $\overline{\Omega} \times [0,T]$ is essential.
		\item Elements of $X_{\nu, k}$ should be uniformly bounded in $\xCn{0}([0,1];\xCn{{\tilde{m},\alpha}}(\overline{\Omega}; \mathbb{R}^2))^2$ by some constant depending only on fixed objects like $\nu$ and $\overline{\mathbf{y}}$. This is satisfied since $\overline{\mathbf{y}}$ is fixed and for $(\mathbf{z}^+, \mathbf{z}^-) \in  X_{\nu, k}$ one has
		\begin{equation}\label{equation:unifbd}
			\begin{aligned}
				\max\limits_{t \in [0,1]} \|\mathbf{z}^{\pm}\|_{\tilde{m},\alpha,\Omega}(t) & \leq && \max\limits_{t \in [0,1]} \omega_k(t) \|\mathbf{z}^{\pm} - \overline{\mathbf{y}} \|_{\tilde{m},\alpha,\Omega}(t) + \max\limits_{t \in [0,1]} \|\overline{\mathbf{y}} \|_{\tilde{m},\alpha,\Omega}(t) \\
				& \leq && \nu + \max\limits_{t \in [0,1]} \| \overline{\mathbf{y}} \|_{\tilde{m},\alpha,\Omega}(t).
			\end{aligned}
		\end{equation}
	\end{itemize}

	\subsection{Construction of a fixed point map} \label{subsection:constructionF}
	On $X_{\nu, k}$ we define a map $F$ by assigning to a given pair $(\overline{\boldsymbol{\mathfrak{z}}^+}, \overline{\boldsymbol{\mathfrak{z}}^-}) \in X_{\nu, k}$ the value $F(\overline{\boldsymbol{\mathfrak{z}}^+}, \overline{\boldsymbol{\mathfrak{z}}^-})$ through the following steps.
	
	\begin{step}\label{step:1}
		Let $\boldsymbol{\mathfrak{z}}^{\pm} = (\mathfrak{z}^{\pm}_1,\mathfrak{z}^{\pm}_2)'$ denote the extensions of $\overline{\boldsymbol{\mathfrak{z}}^{\pm}}$ to $\Omega_3$, defined with the operators from Lemma~\Rref{lemma:extensionOperators} by
		\[
		\begin{cases}
			\boldsymbol{\mathfrak{z}}^+ = \mathbf{y}^* + \pi_2(\overline{\boldsymbol{\mathfrak{z}}^+} - \overline{\mathbf{y}}),\\
			\boldsymbol{\mathfrak{z}}^- = \mathbf{y}^* + \pi_2(\overline{\boldsymbol{\mathfrak{z}}^-} - \overline{\mathbf{y}}).
		\end{cases}
		\]
		In particular, this means that $\boldsymbol{\mathfrak{z}}^{\pm} = \mathbf{y}^*$ in $\Omega_3 \setminus \overline{\Omega}_2$ and $\boldsymbol{\mathfrak{z}}^{\pm} = \overline{\boldsymbol{\mathfrak{z}}^{\pm}}$ in $\Omega$. Moreover, for $(\mathbf{x},s,t) \in \overline{\Omega}_3\times[0,1]\times[0,1]$, the flow maps corresponding to $\boldsymbol{\mathfrak{z}}^{\pm}$ are denoted by $\boldsymbol{\mathcal{Z}}^{\pm}(\mathbf{x},s,t) = (\mathcal{Z}^{\pm}_1(\mathbf{x},s,t),\mathcal{Z}^{\pm}_2(\mathbf{x},s,t))'$. With the help of the Cauchy-Lipschitz Theorem, they are defined through the ordinary differential equations
		\begin{equation}\label{equation:constrFlowOde}
			\begin{cases}
				\xdrv{}{t} \boldsymbol{\mathcal{Z}}^{\pm}(\mathbf{x},s,t) = \boldsymbol{\mathfrak{z}}^{\pm}(\boldsymbol{\mathcal{Z}}^{\pm}(\mathbf{x},s,t),t),\\
				\boldsymbol{\mathcal{Z}}^{\pm}(\mathbf{x},s,s) = \mathbf{x}.
			\end{cases}
		\end{equation}
	\end{step}
	\begin{step}
		Corresponding to $g^{\pm}$ as represented in \eqref{equation:sourceterm_structured_y}, and for the purpose of linearizing \eqref{equation:MHD_curled_Elsaesser}, we denote the following extended versions of the coupling terms by
		\begin{equation}\label{equation:G+-}
			\begin{aligned}
				G^{\pm} & := && -(\partial_1\mathfrak{z}^{\mp}_1 - \partial_1\mathfrak{z}^{\pm}_1)\partial_1(\mathfrak{z}^{\pm}_2 - \mathbf{y}^*_2) -(\partial_1\mathfrak{z}^{\mp}_1 - \partial_1\mathfrak{z}^{\pm}_1)\partial_2(\mathfrak{z}^{\pm}_1 - \mathbf{y}^*_1) \\
				& &&  -(\partial_1\mathfrak{z}^{\pm}_2 - \partial_1\mathfrak{z}^{\mp}_2)\partial_1(\mathfrak{z}^{\pm}_1 - \mathbf{y}^*_1) - (\partial_2\mathfrak{z}^{\pm}_1 - \partial_2\mathfrak{z}^{\mp}_1)\partial_1(\mathfrak{z}^{\pm}_1 - \mathbf{y}^*_1).
			\end{aligned}
		\end{equation}
		By using the continuity of the extension operators from Lemma~\Rref{lemma:extensionOperators}, $G^{\pm}$ obey for all $t \in [0,1]$ the estimate
		\begin{equation}\label{equation:G+-est}
			\begin{aligned}
				\|G^{\pm}\|_{\tilde{m}-1,\alpha,\Omega_3} & \leq && C \|\boldsymbol{\mathfrak{z}}^+ - \boldsymbol{\mathfrak{z}}^-\|_{\tilde{m},\alpha,\Omega}(t) 	\left[\|\boldsymbol{\mathfrak{z}}^+ - \overline{\mathbf{y}}\|_{\tilde{m},\alpha,\Omega}(t) + \|\boldsymbol{\mathfrak{z}}^- - \overline{\mathbf{y}}\|_{\tilde{m},\alpha,\Omega}(t) \right] \\
				& \leq && C	\left[\|\boldsymbol{\mathfrak{z}}^+ - \overline{\mathbf{y}}\|_{\tilde{m},\alpha,\Omega}(t) + \|\boldsymbol{\mathfrak{z}}^- - \overline{\mathbf{y}}\|_{\tilde{m},\alpha,\Omega}(t) \right]^2.
			\end{aligned}
		\end{equation}
		The functions $G^{\pm}$ are related to the representation for $g^{\pm}$ in \eqref{equation:sourceterm_structured_y} only in $\overline{\Omega} \times [0,1]$, where the vector fields $\boldsymbol{\mathfrak{z}}^{\pm}$ are divergence free and the spatial derivatives of $\overline{\mathbf{y}}$ vanish. In $\left( \overline{\Omega_3}\setminus\Omega_2 \right) \times [0,1]$ one observes that $G^{\pm} = 0$, since $\boldsymbol{\mathfrak{z}}^{\pm} = \mathbf{y}^*$ holds there.
	\end{step}
	\begin{step}
		In this step suitable extensions $j^{\pm}_0$ for $\operatorname{curl}(\mathbf{z}^{\pm}_0)$ are introduced in order to guarantee that a fixed point of the map $F$ will vanish at the final time. Let $\tilde{\mathcal{O}} \subseteq \Omega_3$ be an open set with the following properties:
		\begin{enumerate}
			\item $\overline{\tilde{\mathcal{O}}} \subseteq \Omega_3 \setminus \overline{\Omega}_2$.
			\item There exist open $\mathcal{O}^+ \subseteq \tilde{\mathcal{O}}$ and  $\mathcal{O}^- \subseteq \tilde{\mathcal{O}}$ with $\boldsymbol{\mathcal{Z}}^{\pm}(\mathcal{O}^{\pm}, 0, 1) = \Omega_1$.
			\item The set $\tilde{\mathcal{O}}$ is independent of the above choice of $(\overline{\boldsymbol{\mathfrak{z}}^+}, \overline{\boldsymbol{\mathfrak{z}}^-}) \in X_{\nu, k}$, hence it is independent of the fixed point iteration later on.
			\item $\operatorname{dist}(\tilde{\mathcal{O}}, \overline{\Omega}_2 ) \geq c > 0$ for arbitrarily small but fixed  $c > 0$.
		\end{enumerate}
		The existence of a set $\tilde{\mathcal{O}}$ with the properties $(1)$ - $(4)$ is implied by the following points:
		\begin{itemize}
			\item The constant $\nu > 0$ that appears in the definition of $X_{\nu, k}$ is chosen sufficiently small such that Lemma~\Rref{lemma:nu} guarantees that
			\[
			\left[ \boldsymbol{\boldsymbol{\mathcal{Y}}}(\overline{\Omega}_2, 0, 1) \cup \boldsymbol{\mathcal{Z}}^{+}(\overline{\Omega}_2, 0, 1)  \cup \boldsymbol{\mathcal{Z}}^{-}(\overline{\Omega}_2, 0, 1) \right] \cap \overline{\Omega}_2 = \emptyset.
			\]
			\item If every particle that resides in $\overline{\Omega}_2$ at time $t = 0$ is transported by $\boldsymbol{\mathcal{Z}}^{\pm}$ to be at time $t=1$ at a point of $\Omega_3\setminus\overline{\Omega}_2$, then all particles which occupy $\Omega_1$ at time $t = 1$ must originate, due to the continuity of $\boldsymbol{\mathcal{Z}}^{\pm}$, in some open set contained in $\Omega_3\setminus\overline{\Omega}_2$.
			\item By the previous considerations and the fact that $\operatorname{dist}(\Omega_3\setminus\overline{\Omega}_2, \overline{\Omega}_1) > 0$ is a fixed quantity, one can choose $\tilde{\mathcal{O}} \subseteq \Omega_3\setminus\overline{\Omega}_2$ independently of $(\overline{\boldsymbol{\mathfrak{z}}^+}, \overline{\boldsymbol{\mathfrak{z}}^-}) \in X_{\nu, k}$.
		\end{itemize}
		Therefore, one may introduce independently of $(\overline{\boldsymbol{\mathfrak{z}}^+}, \overline{\boldsymbol{\mathfrak{z}}^-}) \in X_{\nu, k}$ a smooth cutoff function $\tilde{\chi} \in \xCn{\infty}_0(\overline{\Omega}_3)$ such that $\tilde{\chi} (\mathbf{x}) = 0$ when $\mathbf{x} \in \Omega_2$ and $\tilde{\chi}(\mathbf{x}) = 1$ for $\mathbf{x} \in \tilde{\mathcal{O}}$. The extended data $j^{\pm}_0$ is then for each $\mathbf{x} \in \overline{\Omega}_3$ defined by
		\begin{equation}\label{equation:constructionF:InitialDataExtension}
			j^{\pm}_0(\mathbf{x}) := \pi_1(\operatorname{curl}(\mathbf{z}^{\pm}_0))(\mathbf{x})
			- \tilde{\chi}(\mathbf{x}) \int_{0}^{1} G^{\pm}(\boldsymbol{\mathcal{Z}}^{\mp}(\mathbf{x},0,\sigma), \sigma) \, d\sigma.
		\end{equation}
		
		The proof for the next auxiliary lemma is postponed until the construction of $F$ is complete.
		\begin{lmm}\label{lemma:composition_estimate}
			For $\sigma \in [0,1]$ and a constant $C(\nu, \mathbf{y}^*) > 0$, which is independent of the choice of $(\overline{\boldsymbol{\mathfrak{z}}^+}, \overline{\boldsymbol{\mathfrak{z}}^-}) \in X_{\nu, k}$, the following estimate is valid:
			\begin{equation}\label{equation:boundcomp}
				\left\| G^{\pm}(\boldsymbol{\mathcal{Z}}^{\mp}(\cdot,0,\sigma), \sigma) \right\|_{\tilde{m}-1,\alpha,\Omega_3} \leq C(\nu, \mathbf{y}^*) \left\| G^{\pm}(\cdot,\sigma) \right\|_{\tilde{m}-1,\alpha,\Omega_3}.
			\end{equation}
		\end{lmm}
		Together with Lemma~\Rref{lemma:composition_estimate}, the extended data $j^{\pm}_0$ defined above in \eqref{equation:constructionF:InitialDataExtension} is seen to satisfy the following properties:
		\begin{itemize}
			\item $j^{\pm}_0 \in \xCn{{\tilde{m}-1,\alpha}}(\Omega_3)$, since
			\begin{equation*}
				\begin{aligned}
					\|j^{\pm}_0\|_{\tilde{m}-1,\alpha,\Omega_3} & \leq && C \int_{0}^{1} \left\| G^{\pm}(\boldsymbol{\mathcal{Z}}^{\mp}(\cdot,0,\sigma), \sigma) \right\|_{\tilde{m}-1,\alpha,\Omega_3} \, d\sigma + C\|\mathbf{z}^{\pm}_0\|_{\tilde{m},\alpha,\Omega}\\
					& \leq && C(\nu, y^*) \int_{0}^{1} \left\| G^{\pm}(\cdot, \sigma) \right\|_{\tilde{m}-1,\alpha,\Omega_3} \, d\sigma + C\|\mathbf{z}^{\pm}_0\|_{\tilde{m},\alpha,\Omega},
				\end{aligned}
			\end{equation*}
			
			where the constant $C > 0$ depends on $\pi_1$ and $\tilde{\chi}$, but not on the choice of $(\overline{\boldsymbol{\mathfrak{z}}^+}, \overline{\boldsymbol{\mathfrak{z}}^-}) \in X_{\nu, k}$.
			\item By the properties of $\pi_1$ from Lemma~\Rref{lemma:extensionOperators} and the cutoff function $\tilde{\chi}$, one has
			\[
			\|j^{\pm}_0\|_{\tilde{m}-1, \alpha, \Omega_2} \leq C \|\mathbf{z}^{\pm}_0\|_{\tilde{m}, \alpha, \Omega}.
			\]
		\end{itemize}
		While the generic constants $C > 0$ or $C(\nu,\mathbf{y}^*) > 0$ can depend on several fixed objects such as the extension operators $\pi_i$, $i\in \{1,2\}$ or $\tilde{m} \geq 3$, we mostly choose to only point out the $\nu$- and $\mathbf{y}^*$- dependence in order to indicate when the properties of $X_{\nu, k}$ are used for obtaining some estimates.
	\end{step}
	\begin{step}\label{step:4}
		Consider the following linear transport equations corresponding to \eqref{equation:MHD_curled_Elsaesser}:
		\begin{equation}\label{equation:constructionF:j+}
			\begin{cases}
				\partial_t j^{\pm} + (\boldsymbol{\mathfrak{z}}^{\mp} \cdot \boldsymbol{\nabla}) j^{\pm} =  G^{\pm}, & \mbox{ in } \Omega_3 \times (0, 1),\\
				j^{\pm}(\cdot, 0)  =  j^{\pm}_0(\cdot), & \mbox{ in } \Omega_3.
			\end{cases}
		\end{equation}
		They admit unique solutions $j^{\pm} \in \xCn{0}([0, 1]; \xCn{{\tilde{m}-1,\alpha}}(\overline{\Omega}_3))$, since $\boldsymbol{\mathfrak{z}}^{\pm}(\mathbf{x}) \cdot \mathbf{n}(\mathbf{x}) = 0$ for all $\mathbf{x} \in \partial\Omega_3$ such that in \eqref{equation:constructionF:j+} no boundary conditions are required.
		
		By the special choice of initial data in \eqref{equation:constructionF:j+}, the unwanted contributions from the source terms $G^{\pm}$ in \eqref{equation:constructionF:j+} are canceled at time $t = 1$. More precisely, for each $\tilde{\mathbf{x}} \in \Omega_1$, due to the definition of the sets $\mathcal{O}^{\pm}$, there exist unique elements $\mathbf{x}^{\pm} \in \mathcal{O}^{\pm}$ with $\boldsymbol{\mathcal{Z}}^{\pm}(\mathbf{x}^{\pm},0, 1) = \tilde{\mathbf{x}}$ and by using the flow maps $\boldsymbol{\mathcal{Z}}^{\pm}$ one has
		\begin{equation}\label{equation:jpmarezero}
			\begin{aligned}
				j^{\pm}(\tilde{\mathbf{x}},1) & = && j^{\pm}(\mathbf{x}^{\mp}, 0) + \int_0^{1} G^{\pm}(\boldsymbol{\mathcal{Z}}^{\mp}(\mathbf{x}^{\mp}, 0, s), s) \, ds \\
				& = && \int_0^{1} G^{\pm}(\boldsymbol{\mathcal{Z}}^{\mp}(\mathbf{x}^{\mp}, 0, s), s) \, ds  - \int_0^{1} G^{\pm}(\boldsymbol{\mathcal{Z}}^{\mp}(\mathbf{x}^{\mp}, 0, s), s) \, ds \\
				& = && 0.
			\end{aligned}
		\end{equation}
		Thus, at time $t=1$ the terminal condition $j^{+}(\mathbf{x},1) = j^{-}(\mathbf{x},1) = 0$ is satisfied for all $\mathbf{x} \in \Omega_1$.
	\end{step}
	\begin{step}\label{step:construct_z}
		Next, the functions $(\tilde{\mathbf{z}}^+, \tilde{\mathbf{z}}^-) \in \xCn{0}([0,1];\xCn{{\tilde{m},\alpha}}(\overline{\Omega}_1; \mathbb{R}^2))^2$ are obtained by solving the $\operatorname{div}$-$\operatorname{curl}$ systems
		\begin{equation}\label{equation:divcurl}
			\begin{aligned}
				\begin{cases}
					\operatorname{curl}(\tilde{\mathbf{z}}^{\pm}) = j^{\pm}, \,\, \operatorname{div}(\tilde{\mathbf{z}}^{\pm}) = \operatorname{div}\left(\lambda \pi_2(\mathbf{z}^{\pm}_0)\right), & \mbox{ in }  \Omega_1 \times [0, 1],\\
					\tilde{\mathbf{z}}^{\pm} \cdot \mathbf{n} = (\mathbf{y}^* + \lambda \pi_2(\mathbf{z}^{\pm}_0)) \cdot \mathbf{n}, & \mbox{ on }  \partial \Omega_1 \times [0,1].
				\end{cases}
			\end{aligned}
		\end{equation}
		Namely, $\tilde{\mathbf{z}}^{\pm}$ are given in $\Omega_1\times(0,1)$ via $\tilde{\mathbf{z}}^{\pm} = \boldsymbol{\nabla}^{\perp}\varphi^{\pm}  + \mathbf{y}^*  + \lambda \pi_2(\mathbf{z}^{\pm}_0)$ where $\varphi^{\pm}(\cdot, t)$ solve the elliptic equations
		\begin{equation}\label{equation:phi}
			\begin{aligned}
				\begin{cases}
					- \Delta \varphi^{\pm}(\cdot, t) = j^{\pm}(\cdot, t) - \lambda(t) \operatorname{curl} (\pi_2(\mathbf{z}^{\pm}_0))(\cdot), & \mbox{ in }  \Omega_1,\\
					\varphi^{\pm}(\cdot, t) = 0, & \mbox{ on }  \partial \Omega_1.
				\end{cases}
			\end{aligned}
		\end{equation}
		The extension $\Omega_1$ was introduced to state the above Poisson equations in a domain with smooth boundary.
	\end{step}
	\begin{step}
		Finally, we set $F(\overline{\boldsymbol{\mathfrak{z}}^+}, \overline{\boldsymbol{\mathfrak{z}}^-}) := (\mathbf{z}^+, \mathbf{z}^-)$, where $\mathbf{z}^+ := \tilde{\mathbf{z}}^+_{|_{\Omega}}$ and $\mathbf{z}^- := \tilde{\mathbf{z}}^-_{|_{\Omega}}$, with $\tilde{\mathbf{z}}^{\pm}$ being the functions constructed in the previous step. Due to the properties of $\mathbf{y}^*$ and $\lambda$, at $t = 1$ one has
		\[
		\begin{cases}
			(\mathbf{y}^*(\cdot, 1) + \lambda(1) \pi_2(\mathbf{z}^{\pm}_0)(\cdot)) \cdot \mathbf{n}(\cdot) = 0, & \mbox{ on } \partial \Omega_1,\\
			\operatorname{div}\left(\lambda(1) \pi_2(\mathbf{z}^{\pm}_0)(\cdot)\right) = 0, & \mbox{ in } \partial \Omega_1,
		\end{cases}
		\]
		and since $j^{\pm}(\cdot, 1) = 0$ in $\Omega_1$ by \eqref{equation:jpmarezero}, one obtains $\tilde{\mathbf{z}}^{\pm}(\cdot, 1) = \mathbf{0}$ in $\Omega_1$ from \eqref{equation:divcurl}, hence $\mathbf{z}^{\pm}(\cdot, 1) = \mathbf{0}$ in $\Omega$. Moreover, we note that $(\Gamma\setminus\Gamma_0) \subseteq \partial \Omega_1$ by definition, thus $\mathbf{z}^{\pm}$ satisfy the boundary conditions of the space $X_{\nu,k}$.
	\end{step}
	
	It remains for this section to show Lemma~\Rref{lemma:composition_estimate}.
	\begin{proof}[Proof of Lemma~\Rref{lemma:composition_estimate}]
		The objective is to verify \eqref{equation:boundcomp} in the context in which Lemma~\Rref{lemma:composition_estimate} is stated. Hereto, the task is to estimate for the composite function $x \mapsto G^{\pm}(\boldsymbol{\mathcal{Z}}^{\mp}(\mathbf{x},0,\sigma), \sigma)$,
		\begin{itemize}
			\item the $\xCn{0}$ norms of all spatial derivatives up to the order $\tilde{m}-1$,
			\item and the $\xCn{{0,\alpha}}$ - seminorm of the ($\tilde{m}-1)$-th order spatial derivatives,
		\end{itemize}
		by upper bounds of the form $C(\nu, \mathbf{y}^*) \left\| G^{\pm}(\cdot,\sigma) \right\|_{\tilde{m}-1,\alpha,\Omega_3}$. Hereto, with the help of a multidimensional Fa\`a di Bruno's Formula (or Arbogast rule), one may write
		\begin{equation}\label{equation:FaaVersion}
			\begin{aligned}
				\partial^{\boldsymbol{\kappa}} G^{\pm}(\boldsymbol{\mathcal{Z}}^{\mp}(\mathbf{x},0,\sigma), \sigma) = \sum_{\hat{\boldsymbol{\pi}} \in P_{\boldsymbol{\kappa}}}\frac{\boldsymbol{\kappa}!}{\hat{\pi}_1!\hat{\pi}_2!} \left[\partial_1^{|\hat{\pi}_1|}\partial_2^{|\hat{\pi}_2|}  G^{\pm} \right](\boldsymbol{\mathcal{Z}}^{\mp}(\mathbf{x},0,\sigma),\sigma) \prod_{l = 1}^{2}\prod_{\boldsymbol{\beta}} \left( \frac{1}{\boldsymbol{\beta}!} \left[ \partial^{\boldsymbol{\beta}}\boldsymbol{\mathcal{Z}}^{\mp}_l \right](\mathbf{x},0,\sigma) \right)^{\hat{\pi}_l(\boldsymbol{\beta})}.
			\end{aligned}
		\end{equation}
		We borrowed this formulation from \cite[Theorem 3.2]{Turcu}, where the notations are explained and proofs can be found. Briefly, $P_{\boldsymbol{\kappa}}$ contains all partitions $\hat{\boldsymbol{\pi}} \colon \mathbb{N}_0^2\setminus\{(0,0)\}\to \mathbb{N}_0^2$ of the multi-index $\boldsymbol{\kappa} \in \mathbb{N}_0^2\setminus\{(0,0)\}$ which satisfy $\sum_{l=1}^2\sum_{\boldsymbol{\beta}} \boldsymbol{\beta} \hat{\pi}_l(\boldsymbol{\beta}) = \boldsymbol{\kappa}$. Moreover, $\hat{\pi}_l! := \prod_{\boldsymbol{\beta}}\hat{\pi}_l(\boldsymbol{\beta})!$ and $|\hat{\pi}_l| := \sum_{\boldsymbol{\beta}}\hat{\pi}_l(\boldsymbol{\beta})$, $l = 1,2$, for  $\hat{\boldsymbol{\pi}} \in P_{\boldsymbol{\kappa}}$.
		
		The formula \eqref{equation:FaaVersion} implies that, aside of combinatorical quantities which can be moved into the constant $C(\nu, \mathbf{y}^*) > 0$, it is enough to estimate two types of terms:
		\begin{itemize}
			\item $E(\cdot,\sigma) := \left[ \partial^{\boldsymbol{\beta}}G^{\pm} \right](\boldsymbol{\mathcal{Z}}^{\mp}(\cdot,0,\sigma), \sigma)$ for all $\boldsymbol{\beta} \in \mathbb{N}_0^{2}$, $|\boldsymbol{\beta}| \leq \tilde{m}-1$.
			\item $\partial^{\boldsymbol{\beta}}\mathcal{Z}_l^{\mp}(\cdot,0,\sigma)$ for all $\boldsymbol{\beta} \in \mathbb{N}_0^{2}$, $|\boldsymbol{\beta}| \leq \tilde{m}-1$ and $l \in \{1,2\}$.
		\end{itemize}
		Hereby, the first kind of term shall always be bounded by $C(\nu, \mathbf{y}^*) \left\| G^{\pm}(\cdot,\sigma) \right\|_{\tilde{m}-1,\alpha,\Omega_3}$, while the second type of term needs to be fully included into $C(\nu, \mathbf{y}^*)$. The key for keeping the constant $C(\nu, \mathbf{y}^*)$ independent of the particular choice of $(\overline{\boldsymbol{\mathfrak{z}}^+}, \overline{\boldsymbol{\mathfrak{z}}^-}) \in X_{\nu, k}$ is that elements of $X_{\nu,k}$, as shown in \eqref{equation:unifbd}, allow a uniform bound depending only on fixed objects such as $\nu$ and $\mathbf{y}^*$.
		
		We begin with the first type of term and obtain for $\boldsymbol{\beta} \in \mathbb{N}_0^{2}$, $|\boldsymbol{\beta}| \leq \tilde{m}-1$, by taking into account that the flows $\boldsymbol{\mathcal{Z}}^{\mp}$ never leave $\Omega_3$, the estimate
		\begin{equation*}
			\begin{aligned}
				\left\|E(\cdot,\sigma)\right\|_{0,\alpha,\Omega_3} & = && \sup\limits_{\substack{\mathbf{x},\mathbf{y} \in \Omega_3 \\ \mathbf{x}\neq \mathbf{y}}} \Bigg\{  \frac{\Big| \left[ \partial^{\boldsymbol{\beta}}G^{\pm} \right](\boldsymbol{\mathcal{Z}}^{\mp}(\mathbf{x},0,\sigma), \sigma)-\left[ \partial^{\boldsymbol{\beta}}G^{\pm} \right](\boldsymbol{\mathcal{Z}}^{\mp}(\mathbf{y},0,\sigma), \sigma)\Big|}{\Big| \boldsymbol{\mathcal{Z}}^{\mp}(\mathbf{x},0,\sigma)-\boldsymbol{\mathcal{Z}}^{\mp}(\mathbf{y},0,\sigma) \Big|^{\alpha}} \frac{\Big| \boldsymbol{\mathcal{Z}}^{\mp}(\mathbf{x},0,\sigma)-\boldsymbol{\mathcal{Z}}^{\mp}(\mathbf{y},0,\sigma) \Big|^{\alpha}}{|\mathbf{x}-\mathbf{y}|^{\alpha}} \Bigg\}\\
				& && + \sup\limits_{\mathbf{x} \in \Omega_3} \left| \left[ \partial^{\boldsymbol{\beta}}G^{\pm} \right](\mathbf{x}, \sigma) \right| \\
				& \leq && \left\| G^{\pm}(\cdot,\sigma) \right\|_{\tilde{m}-1,\alpha,\Omega_3} \left(1 + \sup\limits_{\mathbf{x} \in \Omega_3} \left| \boldsymbol{\nabla} \boldsymbol{\mathcal{Z}}^{\mp}(\mathbf{x},0,\sigma) \right|^{\alpha} \right)\\
				& \leq && C(\nu, \mathbf{y}^*) \left\| G^{\pm}(\cdot,\sigma) \right\|_{\tilde{m}-1,\alpha,\Omega_3}.
			\end{aligned}
		\end{equation*}
		Above, the constant $C(\nu,\mathbf{y}^*) = C(\nu,\mathbf{y}^*, \alpha, \tilde{m}) > 0$ is obtained from the ordinary differential equations \eqref{equation:constrFlowOde} for $\boldsymbol{\mathcal{Z}}^{\mp}$, which lead for $i,l \in \{1,2\}$ to
		\begin{equation}\label{equation:chainruledZ+-Ode}
			\begin{aligned}
				\begin{cases}
					\xdrv{}{t} \partial_l \mathcal{Z}^{\mp}_i(\mathbf{x},0,t) & =  \boldsymbol{\nabla} \mathfrak{z}^{\pm}_i(\boldsymbol{\mathcal{Z}}^{\mp}(\mathbf{x},0,t),t) \cdot \partial_l \boldsymbol{\mathcal{Z}}^{\mp}(\mathbf{x},0,t), \\
					\partial_l \boldsymbol{\mathcal{Z}}^{\pm}(\mathbf{x},0,0) & = \mathbf{e}_l,
				\end{cases}
			\end{aligned}
		\end{equation}
		and which can in turn be integrated in time from $0$ to $\sigma$, such that by using the Gr\"onwall inequality it yields
		\begin{equation}\label{equation:flowgradest}
			\begin{aligned}
				\sup\limits_{\mathbf{x} \in \Omega_3} \Big|\boldsymbol{\nabla}\boldsymbol{\mathcal{Z}}^{\mp}(\mathbf{x},0,\sigma) \Big| & \leq && C\operatorname{exp}\left(\|\boldsymbol{\mathfrak{z}}^{\mp}\|_{1,\alpha,\Omega_3}(\sigma)\right)\\
				& \leq && C\operatorname{exp}\left(\|\boldsymbol{\mathfrak{z}}^{\mp} - \mathbf{y}^*\|_{1,\alpha,\Omega_3}(\sigma) + \|\mathbf{y}^*\|_{1,\alpha,\Omega_3}(\sigma)\right)\\
				& \leq && C \operatorname{exp}\left(\nu + \|\mathbf{y}^*\|_{1,\alpha,\Omega_3}(\sigma)\right).
			\end{aligned}
		\end{equation}
		
		It is left to move the second kind of term from \eqref{equation:FaaVersion} into a constant $C(\nu,\mathbf{y}^*) > 0$. Since $X_{\nu,k}$ provides a uniform bound with respect to the higher norm $\|\cdot\|_{\tilde{m},\alpha,\Omega}$, this can be done by obtaining for all $\tilde{\boldsymbol{\beta}} \in \mathbb{N}_0^{2}$, $|\tilde{\boldsymbol{\beta}}| \leq \tilde{m}$, $l \in \{1,2\}$, estimates of the form
		\begin{equation}\label{equation:oneboundtoshow}
			\sup\limits_{\mathbf{x} \in \Omega_3} \Big| \partial^{\tilde{\boldsymbol{\beta}}}\mathcal{Z}_l^{\mp}(\mathbf{x},0,\sigma) \Big| \leq C(\nu, \mathbf{y}^*).
		\end{equation}
		For this purpose, one can employ again the defining ordinary differential equations for $\boldsymbol{\mathcal{Z}}^{\mp}$, namely
		\begin{equation}\label{equation:deford}
			\begin{aligned}
				\xdrv{}{t}\partial^{\tilde{\boldsymbol{\beta}}}\boldsymbol{\mathcal{Z}}^{\mp}(\mathbf{x},0,\sigma) =  \partial^{\tilde{\boldsymbol{\beta}}} \left[ \boldsymbol{\mathfrak{z}}^{\mp}(\boldsymbol{\mathcal{Z}}^{\mp}(\cdot,0,\sigma),\sigma) \right](\mathbf{x}).
			\end{aligned}
		\end{equation}
		We shall do this inductively:
		\begin{enumerate}
			\item Assume that $|\tilde{\boldsymbol{\beta}}| \leq 1$ in \eqref{equation:oneboundtoshow}. By the Chain Rule, since $|\tilde{\boldsymbol{\beta}}| \leq 1$, \eqref{equation:deford} with initial conditions corresponds either to \eqref{equation:constrFlowOde} or \eqref{equation:chainruledZ+-Ode} and one has by integrating in time
			\begin{equation*}
				\begin{aligned}
					\sup\limits_{\mathbf{x} \in \Omega_3} \sum\limits_{|\hat{\boldsymbol{\beta}}|\leq 1, l = 1,2} \Big| \partial^{\hat{\boldsymbol{\beta}}}\mathcal{Z}_l^{\mp}(\mathbf{x},0,\sigma) \Big| \leq C(\Omega_3) + C(\nu,\mathbf{y}^*)\int_0^{\sigma} \sup\limits_{\mathbf{x} \in \Omega_3} \sum\limits_{|\hat{\boldsymbol{\beta}}|\leq 1, l = 1,2} \Big| \partial^{\hat{\boldsymbol{\beta}}}\mathcal{Z}_l^{\mp}(\mathbf{x},0,s) \Big|^{|\tilde{\boldsymbol{\beta}}|} \, ds,
				\end{aligned}
			\end{equation*}
			where $C(\nu,\mathbf{y}^*) > 0$ includes in particular a bound for all appearing factors of the form
			\[
			\tilde{E}(\cdot, \sigma) := 		\Big|\left[\partial^{\boldsymbol{\beta}^*}\mathfrak{z}_l^{\mp}\right](\boldsymbol{\mathcal{Z}}^{\mp}(\cdot,0,\sigma),\sigma)\Big|,
			\]
			with some $\boldsymbol{\beta}^* \in \mathbb{N}_0^{2}$, $|\boldsymbol{\beta}^*| \leq 1$, since the flows $\boldsymbol{\mathcal{Z}}^{\mp}$ never leave $\Omega_3$ and one can for all $\mathbf{x} \in \overline{\Omega}_3$ estimate
			\[
			\tilde{E}(\mathbf{x}, \sigma) \leq \sup\limits_{\mathbf{x} \in \Omega_3} 		\Big|\left[\partial^{\boldsymbol{\beta}^*}\mathfrak{z}_l^{\mp}\right](\mathbf{x},\sigma)\Big| \leq \nu + \|\mathbf{y}^*\|_{1,\alpha,\Omega_3}.
			\]
			\item Now assume, that for all $|\tilde{\boldsymbol{\beta}}| < \tilde{m}$ a bound as in \eqref{equation:oneboundtoshow} holds. Then, by integrating \eqref{equation:deford} in time and from Fa\`a di Bruno's Formula like in \eqref{equation:FaaVersion}, one has
			\begin{equation}\label{equation:inductionstep2}
				\begin{aligned}
					\sup\limits_{\mathbf{x} \in \Omega_3} \sum\limits_{|\boldsymbol{\kappa}| = \tilde{m}, l = 1,2} \Big| \partial^{\boldsymbol{\kappa}}\mathcal{Z}_l^{\mp}(\mathbf{x},0,\sigma) \Big| & \leq && \int_0^{\sigma} \left[\sup\limits_{\mathbf{x} \in \Omega_3}  \sum\limits_{|\boldsymbol{\kappa}| = \tilde{m}, l = 1,2} V(\mathbf{x},s,\boldsymbol{\kappa},l) \right]\, ds,
				\end{aligned}
			\end{equation}
			where
			\begin{equation*}
				\begin{aligned}
					V(\mathbf{x},s,\boldsymbol{\kappa},l) := \Big| \sum_{\hat{\boldsymbol{\pi}} \in P_{\boldsymbol{\kappa}}}\frac{\boldsymbol{\kappa}!}{\hat{\pi}_1!\hat{\pi}_2!} \left[\partial_1^{|\hat{\pi}_1|}\partial_2^{|\hat{\pi}_2|}  \mathfrak{z}^{\mp}_l \right](\boldsymbol{\mathcal{Z}}^{\mp}(\mathbf{x},0,s),s) \prod_{i = 1}^{2}\prod_{\boldsymbol{\beta}} \left( \frac{1}{\boldsymbol{\beta}!} \left[ \partial^{\boldsymbol{\beta}}\mathcal{Z}^{\mp}_i \right](\mathbf{x},0,s) \right)^{\hat{\pi}_i(\boldsymbol{\beta})} \Big|.
				\end{aligned}
			\end{equation*}
			In the above formula for $V(\mathbf{x},s,\boldsymbol{\kappa},l)$, if $|\boldsymbol{\beta}| = \tilde{m}$, then $\hat{\pi}_i(\boldsymbol{\beta}) = 1$ for $i = 1,2$. This fact comes from the requirement on the partitions $\hat{\boldsymbol{\pi}}$ to satisfy $\sum_{l=1}^2\sum_{\boldsymbol{\beta}} \boldsymbol{\beta} \hat{\pi}_l(\boldsymbol{\beta}) = \boldsymbol{\kappa}$. Therefore, all terms with exponent $\hat{\pi}_l(\boldsymbol{\beta}) > 1$ can be integrated into a constant $C(\nu, \mathbf{y}^*)$ by means of the induction assumption. Moreover, with the same explanation as in the first step of the induction argument, the terms of the form
			\[
			\left[\partial_1^{|\hat{\pi}_1|}\partial_2^{|\hat{\pi}_2|}  \mathfrak{z}^{\mp}_l 	\right](\boldsymbol{\mathcal{Z}}^{\mp}(\mathbf{x},0,s),s)
			\]
			can likewise be estimated by $C(\nu, \mathbf{y}^*)$. In conclusion, \eqref{equation:inductionstep2} simplifies to
			\begin{equation*}
				\begin{aligned}
					\sup\limits_{\mathbf{x} \in \Omega_3} \sum\limits_{|\boldsymbol{\kappa}| = \tilde{m}, l = 1,2} \Big| \partial^{\boldsymbol{\kappa}}\mathcal{Z}_l^{\mp}(\mathbf{x},0,\sigma) \Big| \leq C(\nu, \mathbf{y}^*) + C(\nu, \mathbf{y}^*)\int_0^{\sigma} \left[\sup\limits_{\mathbf{x} \in \Omega_3}  \sum\limits_{|\boldsymbol{\kappa}| = \tilde{m}, l = 1,2} \Big| \partial^{\boldsymbol{\kappa}}\mathcal{Z}_l^{\mp}(\mathbf{x},0,s) \Big| \right]\, ds,
				\end{aligned}
			\end{equation*}
			and using Gr\"onwall's inequality completes the estimate.
		\end{enumerate}
	\end{proof}
	
	\subsection{Existence of a fixed point}\label{subsection:existence}
	The goal of this section is to show that $F(X_{\nu, k}) \subseteq X_{\nu, k}$ and that $F$ is a contraction with respect to the topology of $\xCn{0}([0,1];\xCn{{1,\alpha}}(\overline{\Omega}; \mathbb{R}^2))^2$, which allows then to conclude Proposition~\Rref{proposition:smalldata}. The main steps hereto are stated in Lemma~\Rref{lemma:stepsForContraction} below.
	\begin{lmm}\label{lemma:stepsForContraction}
		There exists a possibly large $k = k(\nu, \mathbf{y}^*) > 0$ and a small number $\delta = \delta(\nu, k, \mathbf{y}^*) > 0$, depending on $\nu, k$ and $\mathbf{y}^*$, such that the following statements are true, provided that $\|\mathbf{z}^+_0\|_{\tilde{m},\alpha,\Omega} + \|\mathbf{z}^-_0\|_{\tilde{m},\alpha,\Omega} < \delta$:
		\begin{enumerate}
			\item $F(X_{\nu, k}) \subseteq X_{\nu, k}$.
			\item There exists a constant $\kappa \in (0,1)$ such that for all $(\overline{\boldsymbol{\mathfrak{z}}^{+,1}}, \overline{\boldsymbol{\mathfrak{z}}^{-,1}}), (\overline{\boldsymbol{\mathfrak{z}}^{+,2}}, \overline{\boldsymbol{\mathfrak{z}}^{-,2}}) \in X_{\nu, k}$ one has
			\[
			\|F(\overline{\boldsymbol{\mathfrak{z}}^{+,1}}, \overline{\boldsymbol{\mathfrak{z}}^{-,1}}) - F(\overline{\boldsymbol{\mathfrak{z}}^{+,2}}, \overline{\boldsymbol{\mathfrak{z}}^{-,2}}) \|_{X} \leq \kappa \|(\overline{\boldsymbol{\mathfrak{z}}^{+,1}}, \overline{\boldsymbol{\mathfrak{z}}^{-,1}}) - (\overline{\boldsymbol{\mathfrak{z}}^{+,2}}, \overline{\boldsymbol{\mathfrak{z}}^{-,2}})\|_{X},
			\]
			with the notation
			\[
			\|(\mathbf{f}, \mathbf{g})\|_{X} := \max_{t \in [0,1]} \|(\mathbf{f}, \mathbf{g})\|_{\xCn{{1,\alpha}}(\overline{\Omega}; \mathbb{R}^2)^2}(t) = \max_{t \in [0,1]} \left\{ \|\mathbf{f}\|_{1,\alpha,\Omega}(t) + \|\mathbf{g}\|_{1,\alpha,\Omega}(t) \right\}.
			\]
		\end{enumerate}
	\end{lmm}
	
	The proof of Lemma~\Rref{lemma:stepsForContraction} shall be given in Section~\Rref{section:proofSTepsContr} below. First, we conclude Proposition~\Rref{proposition:smalldata} under the assumption that  Lemma~\Rref{lemma:stepsForContraction} holds. Indeed, then $F$ is a contraction and uniformly continuous on $X_{\nu, k}$ with respect to $\|\cdot\|_{X}$. Therefore $F$ can be uniquely extended to a continuous map $\tilde{F} \colon \overline{X_{\nu, k}} \to \overline{X_{\nu, k}}$, where $\overline{X_{\nu, k}}$ denotes the closure of $X_{\nu, k}$ with respect to $\|\cdot\|_{X}$. Moreover, $\tilde{F}$ is a self-map of $\overline{X_{\nu, k}}$ and contractive in the norm $\|\cdot\|_{X}$. Hence, with the help of Banach's Fixed Point Theorem, $\tilde{F}$ admits a unique fixed point $(\mathbf{z}^+, \mathbf{z}^-) \in \overline{X_{\nu, k}}$, see also \cite[Theorem 3]{BoundaryControl_EulerBoussinesq_Cara_Santos_Souza}. By the construction of $F$ and the definition of the set $X_{\nu, k}$, this fixed point is a solution to \eqref{equation:MHD_curled_Elsaesser} with $\mathbf{z}^+(\mathbf{x}, 1) = \mathbf{z}^-(\mathbf{x}, 1) = \mathbf{0}$ for all $\mathbf{x} \in \Omega$. In order to verify $\mathbf{z}^{\pm} \in \xLinfty([0,T];\xCn{{\tilde{m},\alpha}}(\overline{\Omega}; \mathbb{R}^2))$, let $(\mathbf{z}^+_n, \mathbf{z}^-_n)_{n \in \mathbb{N}} \subseteq X_{\nu, k}$ be a sequence obeying
	\[
	(\mathbf{z}^+_n, \mathbf{z}^-_n)_n \to (\mathbf{z}^+, \mathbf{z}^-) \mbox{ in } \|\cdot\|_X, \, \mbox{ as } n \to +\infty.
	\]
	Then, by the definition of the set $X_{\nu, k}$, there is a uniform bound $R > 0$ such that for all $n \in \mathbb{N}$ it holds
	\[
	\max_{t \in [0,1]} \|\mathbf{z}^+_n\|_{\tilde{m},\alpha,\Omega}(t) + \max_{t \in [0,1]} \|\mathbf{z}^-_n\|_{\tilde{m},\alpha,\Omega}(t) \leq R,
	\]
	which yields $(\mathbf{z}^+, \mathbf{z}^-) \in \xLinfty([0,T];\xCn{{\tilde{m},\alpha}}(\overline{\Omega}; \mathbb{R}^2))^2$.
	
	\subsection{Proof of Lemma~\Rref{lemma:stepsForContraction} (1): $F$ is a self map of $X_{\nu,k}$}\label{section:proofSTepsContr}
	Let $(\overline{\boldsymbol{\mathfrak{z}}^{+}}, \overline{\boldsymbol{\mathfrak{z}}^{-}}) \in X_{\nu, k}$ be arbitrarily fixed and based on this choice denote $\boldsymbol{\mathfrak{z}}^{\pm}$, $\boldsymbol{\mathcal{Z}}^{\pm}$, $j^{\pm}$, $\varphi^{\pm}$, $\tilde{\mathbf{z}}^{\pm}$ and $\mathbf{z}^{\pm}$ from the construction of $F$ as previously. In particular, \eqref{equation:phi} implies for each time $t \in [0,1]$ that
	\begin{equation}\label{equation:curlest}
		\begin{cases}
			\operatorname{div}(\boldsymbol{\nabla}^{\perp}\varphi^{\pm}) = 0, & \mbox{ in } \Omega_1,\\
			\boldsymbol{\nabla}^{\perp}\varphi^{\pm} \cdot \mathbf{n} = \partial_{\boldsymbol{\tau}} \varphi^{\pm} = 0, & \mbox{ on } \partial \Omega_1,
		\end{cases}
	\end{equation}
	and thus $\|\boldsymbol{\nabla}^{\perp}\varphi^{\pm} \|_{\tilde{m},\alpha,\Omega_1}(t) \leq C \|\operatorname{curl}(\boldsymbol{\nabla}^{\perp}\varphi^{\pm}) \|_{\tilde{m}-1,\alpha,\Omega_1}(t)$. Accordingly, for all $t \in [0,1]$ one can deduce that
	\begin{equation}\label{equation:smallerMuEstimate}
		\begin{aligned}
			\omega_k(t) \|\mathbf{z}^{\pm} - \overline{\mathbf{y}}\|_{\tilde{m},\alpha,\Omega}(t) & = && \omega_k(t) \|\tilde{\mathbf{z}}^{\pm} - \mathbf{y}^*\|_{\tilde{m},\alpha,\Omega}(t) \\
			& \leq && \omega_k(t)\|\boldsymbol{\nabla}^{\perp}\varphi^{\pm} \|_{\tilde{m},\alpha,\Omega_1}(t) + \omega_k(t)\|\mathbf{z}^{\pm}_0\|_{\tilde{m},\alpha,\Omega}\\
			& \leq &&  C\omega_k(t)\|\operatorname{curl}(\boldsymbol{\nabla}^{\perp}\varphi^{\pm}) \|_{\tilde{m}-1,\alpha,\Omega_1}(t) + \omega_k(t)\|\mathbf{z}^{\pm}_0\|_{\tilde{m},\alpha,\Omega}\\
			& \leq && 	C \omega_k(t) \left[ \| j^{\pm} \|_{\tilde{m}-1, \alpha, \Omega_1}(t) + \|\mathbf{z}^{\pm}_0\|_{\tilde{m},\alpha,\Omega} \right].
		\end{aligned}
	\end{equation}
	In order to estimate $\|j^{\pm}\|_{\tilde{m}-1, \alpha, \Omega_1}(t)$ for $t \in [0, 1]$, one can employ Remark~\Rref{remark:Transport_Estimate} and obtains
	\begin{equation}\label{equation:auxilliary:bound_j+_a}
		\begin{aligned}
			\|j^{\pm} \|_{\tilde{m}-1,\alpha,\Omega_1}(t) & \leq && \|j^{\pm} \|_{\tilde{m}-1,\alpha,\Omega_3} (t)  \\
			& \leq && \operatorname{exp}\left(C \int_0^t \|\boldsymbol{\mathfrak{z}}^{\mp}\|_{\tilde{m}-1, \alpha, \Omega_3}(s)  \, ds \right) \left(\int_0^t \|G^{\pm}\|_{\tilde{m}-1,\alpha, \Omega_3}(s) \, ds + \|j^{\pm}_0\|_{\tilde{m}-1,\alpha,\Omega_3} \right)  \\
			& \leq && C(\nu, \mathbf{y}^*)\left[ \int_0^t \|G^{\pm}\|_{\tilde{m}-1,\alpha, \Omega_3}(s) \, ds + \|j^{\pm}_0\|_{\tilde{m}-1,\alpha,\Omega_3} \right],
		\end{aligned}
	\end{equation}
	where $C(\nu, \mathbf{y}^*) > 0$ includes the terms coming from the estimate
	\begin{equation*}
		\begin{aligned}
			\operatorname{exp}\left(C \int_0^t \|\boldsymbol{\mathfrak{z}}^{\mp}\|_{\tilde{m}-1, \alpha, \Omega_3}(s)  \, ds \right) & \leq && \operatorname{exp}\left(C \int_0^t \left[ \|\boldsymbol{\mathfrak{z}}^{\mp} - \overline{\mathbf{y}}\|_{\tilde{m}-1, \alpha, \Omega}(s) + \|\overline{\mathbf{y}}\|_{\tilde{m}-1, \alpha, \Omega}(s) \right] \, ds \right) \\
			& \leq && \operatorname{exp}\left(C \int_0^t \left[ \nu + \|\overline{\mathbf{y}}\|_{\tilde{m}-1, \alpha, \Omega}(s) \right] \, ds \right).
		\end{aligned}
	\end{equation*}
	Thus, by inserting the definition of $j^{\pm}_0$ in \eqref{equation:constructionF:InitialDataExtension} into \eqref{equation:auxilliary:bound_j+_a} and applying Lemma~\Rref{lemma:composition_estimate} one gets
	\begin{equation}\label{equation:auxilliary:bound_j+_b}
		\begin{aligned}
			\|j^{\pm} \|_{\tilde{m}-1,\alpha,\Omega_3}(t)& \leq && C(\nu, \mathbf{y}^*)\left[ \int_0^t \|G^{\pm}\|_{\tilde{m}-1,\alpha, \Omega_3}(s) \, ds + \|\mathbf{z}^{\pm}_0\|_{\tilde{m},\alpha,\Omega} + \int_0^1 \|G^{\pm}\|_{\tilde{m}-1,\alpha, \Omega_3}(s) \, ds \right]  \\
			& \leq && C(\nu, \mathbf{y}^*) \int_0^{t} \frac{\omega_k(s)^2}{\omega_k(s)^2} \left[\|\boldsymbol{\mathfrak{z}}^{\pm} - \overline{\mathbf{y}}\|_{\tilde{m},\alpha, \Omega}(s) + \|\boldsymbol{\mathfrak{z}}^{\mp} - \overline{\mathbf{y}}\|_{\tilde{m},\alpha, \Omega}(s)  \right]^2 \, ds \\
			& &&  + C(\nu, \mathbf{y}^*) \int_0^{1} \frac{\omega_k(s)^2}{\omega_k(s)^2} \left[\|\boldsymbol{\mathfrak{z}}^{\pm} - \overline{\mathbf{y}}\|_{\tilde{m},\alpha, \Omega}(s) + \|\boldsymbol{\mathfrak{z}}^{\mp} - \overline{\mathbf{y}}\|_{\tilde{m},\alpha, \Omega}(s)  \right]^2 \, ds  + C(\nu, \mathbf{y}^*) \|\mathbf{z}^{\pm}_0\|_{\tilde{m},\alpha,\Omega} \\
			& \leq && C(\nu, \mathbf{y}^*) \int_0^{1} \frac{1}{\omega_k(s)^2} \, ds + C(\nu, \mathbf{y}^*) \|\mathbf{z}^{\pm}_0\|_{\tilde{m},\alpha,\Omega}.
		\end{aligned}
	\end{equation}
	Next, for all $t \in [0,1]$, the estimates \eqref{equation:weight_trick}, \eqref{equation:smallerMuEstimate} and \eqref{equation:auxilliary:bound_j+_b} yield
	\begin{equation*}\label{equation:smallerMuEstimateb}
		\begin{aligned}
			\omega_k(t) \|\mathbf{z}^{\pm} - \overline{\mathbf{y}}\|_{\tilde{m},\alpha,\Omega}(t) & \leq && C(\nu, \mathbf{y}^*) \omega_k(t) \int_0^{1} \frac{1}{\omega_k(s)^2} \, ds + C(\nu, \mathbf{y}^*) \omega_k(t) \|\mathbf{z}^{\pm}_0\|_{\tilde{m}, \alpha, \Omega}  \\
			& \leq &&  \frac{C(\nu, \mathbf{y}^*)}{2k+1} + C(\nu, \mathbf{y}^*) \omega_k(t) \|\mathbf{z}^{\pm}_0\|_{\tilde{m}, \alpha, \Omega}.
		\end{aligned}
	\end{equation*}
	Thus, provided that $k$ is sufficiently large depending on $\nu$ and $\mathbf{y}^*$ and as long as $\|\mathbf{z}^{\pm}_0\|_{\tilde{m},\alpha,\Omega}$ is sufficiently small, depending on $\nu, \mathbf{y}^*$ as well as on $k$, one arrives at
	\begin{equation*}\label{equation:smallerMuEstimatec}
		\begin{aligned}
			\omega_k(t) \|\mathbf{z}^{\pm} - \overline{\mathbf{y}}\|_{\tilde{m},\alpha,\Omega}(t) < \nu,
		\end{aligned}
	\end{equation*}
	and $F(X_{\nu, k}) \subseteq X_{\nu, k}$ is proved. \qed
	
	\subsection{Proof of Lemma~\Rref{lemma:stepsForContraction} (2): $F$ is contractive with respect to $\|\cdot\|_X$}\label{subsection:contr}
	Following the construction of $F$ above, for $i = 1, 2$, we assign to arbitrarily fixed pairs $(\overline{\boldsymbol{\mathfrak{z}}^{+,i}}, \overline{\boldsymbol{\mathfrak{z}}^{-,i}}) \in X_{\nu, k}$ the corresponding extensions $\boldsymbol{\mathfrak{z}}^{\pm,i}$ and then according to \eqref{equation:G+-}, \eqref{equation:constructionF:j+} as well as \eqref{equation:phi} the functions $G^{\pm,i}$, $j^{\pm,i}$, $\varphi^{\pm,i}$, $\mathbf{z}^{\pm,i}$. In particular, the differences $J^{\pm} := j^{\pm,1}-j^{\pm,2}$ satisfy
	\begin{equation}\label{equation:J+-}
		\begin{cases}
			\partial_t J^{\pm} + (\boldsymbol{\mathfrak{z}}^{\mp,1} \cdot \boldsymbol{\nabla}) J^{\pm} + ((\boldsymbol{\mathfrak{z}}^{\mp,1}-\boldsymbol{\mathfrak{z}}^{\mp,2}) \cdot \boldsymbol{\nabla})j^{\pm,2}  = G^{\pm, 1} - G^{\pm, 2}, & \mbox{ in } \Omega_3 \times (0, 1),\\
			J^{\pm} (\cdot, 0)  =  j^{\pm,1}_0(\cdot)-j^{\pm,2}_0(\cdot), & \mbox{ in } \Omega_3,
		\end{cases}
	\end{equation}
	whereas $j^{\pm,1}_0-j^{\pm,2}_0$ vanish only in $\Omega_2$ but not necessarily in $\Omega_3\setminus\overline{\Omega}_2$, and where
	\begin{equation}\label{equation:GDifference}
		\begin{aligned}
			G^{\pm, 1} - G^{\pm, 2} & = && -(\partial_1 \mathfrak{z}_1^{\mp,1} - \partial_1 \mathfrak{z}_1^{\pm,1}) \left[ (\partial_1 \mathfrak{z}_2^{\pm,1} - \partial_1 \mathfrak{z}_2^{\pm,2}) + (\partial_2 \mathfrak{z}_1^{\pm,1} - \partial_2 \mathfrak{z}_1^{\pm,2}) \right] \\
			& && -(\partial_1 \mathfrak{z}_2^{\pm,2} - \partial_1\mathbf{y}^*_2) \left[ (\partial_1 \mathfrak{z}_1^{\mp,1} - \partial_1 \mathfrak{z}_1^{\mp,2}) + (\partial_1 \mathfrak{z}_1^{\pm,2} - \partial_1 \mathfrak{z}_1^{\pm,1}) \right]\\
			& && -(\partial_2 \mathfrak{z}_1^{\pm,2} - \partial_2\mathbf{y}^*_1) \left[ (\partial_1 \mathfrak{z}_1^{\mp,1} - \partial_1 \mathfrak{z}_1^{\mp,2}) + (\partial_1 \mathfrak{z}_1^{\pm,2} - \partial_1 \mathfrak{z}_1^{\pm,1}) \right]\\
			& && -(\partial_1 \mathfrak{z}_1^{\pm,1} - \partial_1\mathbf{y}^*_1) \left[ (\partial_1 \mathfrak{z}_2^{\pm,1} - \partial_1 \mathfrak{z}_2^{\pm,2}) + (\partial_1 \mathfrak{z}_2^{\mp,2} - \partial_1 \mathfrak{z}_2^{\mp,1}) \right]\\
			& && -(\partial_1 \mathfrak{z}_1^{\pm,1} - \partial_1\mathbf{y}^*_1) \left[ (\partial_2 \mathfrak{z}_1^{\pm,1} - \partial_2 \mathfrak{z}_1^{\pm,2}) + (\partial_2 \mathfrak{z}_1^{\mp,2} - \partial_2 \mathfrak{z}_1^{\mp,1}) \right]\\
			& && + (\partial_1 \mathfrak{z}_1^{\pm,2} - \partial_1 \mathfrak{z}_1^{\pm,1}) \left[ \partial_1 \mathfrak{z}_2^{\pm,2} - \partial_1 \mathfrak{z}_2^{\mp,2}\right]\\
			& && + (\partial_1 \mathfrak{z}_1^{\pm,2} - \partial_1 \mathfrak{z}_1^{\pm,1}) \left[ \partial_2 \mathfrak{z}_1^{\pm,2} - \partial_2 \mathfrak{z}_1^{\mp,2}\right].
		\end{aligned}
	\end{equation}
	From the definition of $X_{\nu, k}$, since each term in \eqref{equation:GDifference} contains a factor that can be bounded by $2\|\boldsymbol{\mathfrak{z}}^{\pm,i} - \overline{\mathbf{y}}\|_{1,\alpha,\Omega}$, one has with a constant $K_1 > 0$, which is independent of $(\overline{\boldsymbol{\mathfrak{z}}^{+,i}}, \overline{\boldsymbol{\mathfrak{z}}^{-,i}}) \in X_{\nu, k}$, $i=1,2$, for all $t \in [0,1]$ that
	\begin{equation}\label{equation:diffestG+-_1_2}
		\begin{aligned}
			\omega_k(t)\|G^{\pm, 1} - G^{\pm, 2}\|_{0,\alpha, \Omega_3}(t) \leq \nu K_1\left( \|\boldsymbol{\mathfrak{z}}^{+,1} - \boldsymbol{\mathfrak{z}}^{+,2}\|_{1,\alpha, \Omega}(t) + \|\boldsymbol{\mathfrak{z}}^{-,1} - \boldsymbol{\mathfrak{z}}^{-,2}\|_{1,\alpha, \Omega}(t) \right).
		\end{aligned}
	\end{equation}
	
	The goal is now to estimate for each $t \in [0,1]$ the quantity
	\[
	D(t) := \|\mathbf{z}^{+,1} - \mathbf{z}^{+,2}\|_{1,\alpha, \Omega}(t) + \|\mathbf{z}^{-,1} - \mathbf{z}^{-,2}\|_{1,\alpha, \Omega}(t).
	\]
	Hereto, from \eqref{equation:J+-}, using the idea of \eqref{equation:curlest} and Remark~\Rref{remark:Transport_Estimate}, one has
	\begin{equation}\label{equation:contr_estimate1}
		\begin{aligned}
			D(t) & \leq && C \left( \|\boldsymbol{\nabla}^{\perp}\varphi^{+, 1} - \boldsymbol{\nabla}^{\perp}\varphi^{+, 2}\|_{1,\alpha, \Omega_1}(t) + \|\boldsymbol{\nabla}^{\perp}\varphi^{-, 1} - \boldsymbol{\nabla}^{\perp}\varphi^{-, 2}\|_{1,\alpha, \Omega_1}(t) \right) \\
			& \leq && C \left( \|J^+\|_{0,\alpha, \Omega_1}(t) + \|J^-\|_{0,\alpha, \Omega_1}(t) \right)\\
			& \leq && C \left( \|J^+\|_{0,\alpha, \Omega_3}(t) + \|J^-\|_{0,\alpha, \Omega_3}(t) \right)\\
			& \leq && I_1^+(t) + I_1^-(t) + I_2^+(t) + I_2^-(t) + I_3^+(t) + I_3^-(t),
		\end{aligned}
	\end{equation}
	where
	\begin{equation*}
		\begin{aligned}
			I_1^{\pm}(t) & := && \tilde{C}(\nu, \mathbf{y}^*) \int_0^t  \|\boldsymbol{\mathfrak{z}}^{\mp,1} -  \boldsymbol{\mathfrak{z}}^{\mp,2}\|_{0,\alpha, \Omega}(s) \|j^{\pm,2}\|_{1,\alpha, \Omega_3}(s)\, ds,\\
			I_2^{\pm}(t) & : = && \tilde{C}(\nu, \mathbf{y}^*) \int_0^t  \|G^{\pm,1} -  G^{\pm,2}\|_{0,\alpha, \Omega_3}(s) \, ds, \\
			I_3^{\pm}(t) & := && \tilde{C}(\nu, \mathbf{y}^*) \int_{0}^{1} \left\| G^{\pm,1}(\boldsymbol{\mathcal{Z}}^{\mp,1}(\cdot,0,s), s)-G^{\pm,2}(\boldsymbol{\mathcal{Z}}^{\mp,2}(\cdot,0,s), s) \right\|_{0,\alpha,\Omega_3} \, ds.
		\end{aligned}
	\end{equation*}
	Hereby, in $I_3^{\pm}(t)$ the extended initial data defined in \eqref{equation:constructionF:InitialDataExtension} has been taken into account. Let from now on $\kappa \in (0,1)$ be an arbitrary but fixed constant. Then, with $K_2 := \max\{1, C(\nu, \mathbf{y}^*), \tilde{C}(\nu, \mathbf{y}^*)\}$, where $C(\nu, \mathbf{y}^*)$ is the constant from the last line of \eqref{equation:auxilliary:bound_j+_b}, the terms $I_1^{\pm}(t)$ obey
	\begin{equation}\label{equation:contr_estimate2}
		\begin{aligned}
			I_1^{\pm}(t) & \leq && K_2^2 \left[\int_0^{1} \frac{1}{\omega_k(s)^2} \, ds + \|\mathbf{z}^{\pm}_0\|_{2,\alpha,\Omega}\right] \max\limits_{t \in [0,1]}  \|\boldsymbol{\mathfrak{z}}^{\mp,1} -  \boldsymbol{\mathfrak{z}}^{\mp,2}\|_{0,\alpha, \Omega}(t) \\
			& \leq && \frac{\kappa}{6} \max\limits_{t \in [0,1]} \|\overline{\boldsymbol{\mathfrak{z}}^{\mp,1}} -  \overline{\boldsymbol{\mathfrak{z}}^{\mp,2}}\|_{1,\alpha, \Omega}(t).
		\end{aligned}
	\end{equation}
	For obtaining \eqref{equation:contr_estimate2}, we may have increased $k$ and reduced $\|\mathbf{z}^{\pm}_0\|_{2,\alpha,\Omega}$, which is justified since $K_2$ does not depend on the choices for $\overline{\boldsymbol{\mathfrak{z}}^{+,i}}$ and $\overline{\boldsymbol{\mathfrak{z}}^{-,i}}$, $i = 1, 2$, but only on fixed objects such as the domain, $\nu$ and $\mathbf{y}^*$. Next, for estimating the two terms $I_2^{\pm}(t)$, one can use \eqref{equation:diffestG+-_1_2} in order to obtain for all $t \in [0,1]$ that
	\begin{equation}\label{equation:contr_estimate3}
		\begin{aligned}
			I_2^{\pm}(t) & = && \tilde{C}(\nu, \mathbf{y}^*) \int_0^t \frac{\omega_k(s)}{\omega_k(s)} \|G^{\pm,1} -  G^{\pm,2}\|_{0,\alpha, \Omega_3}(s) \, ds \\
			& \leq && K_1K_2 \left[\int_0^{1} \frac{1}{\omega_k(s)} \, ds \right] \max\limits_{t \in [0,1]} \left\{ \|\boldsymbol{\mathfrak{z}}^{+,1} -  \boldsymbol{\mathfrak{z}}^{+,2}\|_{1,\alpha, \Omega}(t) + \|\boldsymbol{\mathfrak{z}}^{-,1} -  \boldsymbol{\mathfrak{z}}^{-,2}\|_{1,\alpha, \Omega}(t)\right\} \\
			& \leq && \frac{\kappa}{6} \max\limits_{t \in [0,1]} \left\{ \|\overline{\boldsymbol{\mathfrak{z}}^{+,1}}-  \overline{\boldsymbol{\mathfrak{z}}^{+,2}}\|_{1,\alpha, \Omega}(t) + \|\overline{\boldsymbol{\mathfrak{z}}^{-,1}} -  \overline{\boldsymbol{\mathfrak{z}}^{-,2}}\|_{1,\alpha, \Omega}(t)\right\},
		\end{aligned}
	\end{equation}
	where we may again have increased $k$ and reduced $\|\mathbf{z}^{\pm}_0\|_{2,\alpha,\Omega}$. Since the calculations regarding $I_3^{\pm}(t)$ are longer, let us assume at first the following bounds to be true as well:
	\begin{lmm}\label{lemma:auxineq}
		After possibly increasing $k$ and then reducing $\|\mathbf{z}^{\pm}_0\|_{2,\alpha,\Omega}$, for all $t \in [0,1]$ the terms $I_3^{\pm}(t)$ satisfy
		\[
		I_3^{\pm}(t) \leq \frac{\kappa}{6} \max\limits_{t \in [0,1]} \left\{ \|\overline{\boldsymbol{\mathfrak{z}}^{+,1}} -  \overline{\boldsymbol{\mathfrak{z}}^{+,2}}\|_{1,\alpha, \Omega}(t) + \|\overline{\boldsymbol{\mathfrak{z}}^{-,1}} -  \overline{\boldsymbol{\mathfrak{z}}^{-,2}}\|_{1,\alpha, \Omega}(t)\right\}.
		\]
	\end{lmm}
	Under the assumption that Lemma~\Rref{lemma:auxineq} is valid, one can sum up all the estimates for $I_1^{\pm}$, $I_2^{\pm}$, $I_3^{\pm}$ and complete the proof for Lemma~\Rref{lemma:stepsForContraction} (2) by observing that
	\begin{equation*}
		\begin{aligned}
			\|F(\overline{\boldsymbol{\mathfrak{z}}^{+,1}}, \overline{\boldsymbol{\mathfrak{z}}^{-,1}}) - F(\overline{\boldsymbol{\mathfrak{z}}^{+,2}}, \overline{\boldsymbol{\mathfrak{z}}^{-,2}}) \|_{X} & = && \max\limits_{t \in [0,1]} D(t) \\
			& \leq && \kappa \max\limits_{t \in [0,1]} \left\{ \|\overline{\boldsymbol{\mathfrak{z}}^{+,1}} -  \overline{\boldsymbol{\mathfrak{z}}^{+,2}}\|_{1,\alpha, \Omega}(t) + \|\overline{\boldsymbol{\mathfrak{z}}^{-,1}} -  \overline{\boldsymbol{\mathfrak{z}}^{-,2}}\|_{1,\alpha, \Omega}(t)\right\} \\
			& = && \kappa \|(\overline{\boldsymbol{\mathfrak{z}}^{+,1}}, \overline{\boldsymbol{\mathfrak{z}}^{-,1}}) - (\overline{\boldsymbol{\mathfrak{z}}^{+,2}}, \overline{\boldsymbol{\mathfrak{z}}^{-,2}})\|_{X}.
		\end{aligned}
	\end{equation*}\qed
	
	It is left to verify Lemma~\Rref{lemma:auxineq}.
	\begin{proof}[Proof of Lemma~\Rref{lemma:auxineq}]
		We employ for every $t \in [0,1]$ the estimate
		\begin{equation}\label{equation:G+-_1_2_decomp}
			\begin{aligned}
				\left\| G^{\pm,1}(\boldsymbol{\mathcal{Z}}^{\mp,1}(\cdot,0,t), t)-G^{\pm,2}(\boldsymbol{\mathcal{Z}}^{\mp,2}(\cdot,0,t), t) \right\|_{0,\alpha,\Omega_3} \leq A^{\pm}(t) + B^{\pm}(t),
			\end{aligned}
		\end{equation}
		where
		\[
		\begin{cases}
			A^{\pm}(t) := \left\| G^{\pm,1}(\boldsymbol{\mathcal{Z}}^{\mp,1}(\cdot,0,t), t)-G^{\pm,1}(\boldsymbol{\mathcal{Z}}^{\mp,2}(\cdot,0,t), t)\right\|_{0,\alpha,\Omega_3},\\
			B^{\pm}(t) := \left\|G^{\pm,1}(\boldsymbol{\mathcal{Z}}^{\mp,2}(\cdot,0,t), t) - G^{\pm,2}(\boldsymbol{\mathcal{Z}}^{\mp,2}(\cdot,0,t), t) \right\|_{0,\alpha,\Omega_3}.
		\end{cases}
		\]
		
		We begin with treating $B^{\pm}(t)$ and thereto write $\mathcal{G}^{\pm}(x,t) := G^{\pm,1}(\mathbf{x}, t) - G^{\pm,2}(\mathbf{x}, t)$. Since the flows $\boldsymbol{\mathcal{Z}}^{\mp,2}$ always remain in $\Omega_3$, one can estimate
		\begin{equation}\label{equation:B+-0}
			\begin{aligned}
				B^{\pm}(t)  & \leq &&  \sup\limits_{\substack{\mathbf{x},\mathbf{y} \in \Omega_3 \\ \mathbf{x}\neq \mathbf{y}}} \Bigg\{  \frac{\Big| \left[ \mathcal{G}^{\pm} \right](\boldsymbol{\mathcal{Z}}^{\mp,2}(\mathbf{x},0,t), t)-\left[ \mathcal{G}^{\pm} \right](\boldsymbol{\mathcal{Z}}^{\mp,2}(\mathbf{y},0,t), t)\Big|}{\Big| \boldsymbol{\mathcal{Z}}^{\mp,2}(\mathbf{x},0,t)-\boldsymbol{\mathcal{Z}}^{\mp,2}(\mathbf{y},0,t) \Big|^{\alpha}} \frac{\Big| \boldsymbol{\mathcal{Z}}^{\mp,2}(\mathbf{x},0,t)-\boldsymbol{\mathcal{Z}}^{\mp,2}(\mathbf{y},0,t) \Big|^{\alpha}}{|\mathbf{x}-\mathbf{y}|^{\alpha}} \Bigg\} \\
				& &&  + \sup\limits_{\mathbf{x} \in \Omega_3} \left| \left[ \mathcal{G}^{\pm} \right](\mathbf{x}, \sigma) \right|\\
				& \leq && \left\| \mathcal{G}^{\pm}(\cdot,\sigma) \right\|_{0,\alpha,\Omega_3} \left(1 + \sup\limits_{\mathbf{x} \in \Omega_3} \left| \boldsymbol{\nabla} \boldsymbol{\mathcal{Z}}^{\mp,2}(\mathbf{x},0,t) \right|^{\alpha} \right).
			\end{aligned}
		\end{equation}
		Then, with the aid of \eqref{equation:diffestG+-_1_2} and considerations similar to \eqref{equation:flowgradest}, it is possible to choose a constant $K_3 = K_3(\nu, \mathbf{y}^*) > 0$, which does not depend on $(\overline{\boldsymbol{\mathfrak{z}}^{+,i}}, \overline{\boldsymbol{\mathfrak{z}}^{-,i}}) \in X_{\nu, k}$, $i = 1,2$, such that
		\begin{equation}\label{equation:B+-1}
			\begin{aligned}
				\omega_k(t) B^{\pm}(t) \leq  K_3 \max\limits_{t \in [0,1]}\left\{\|\boldsymbol{\mathfrak{z}}^{+,1} - \boldsymbol{\mathfrak{z}}^{+,2}\|_{1, \alpha}(t)+\|\boldsymbol{\mathfrak{z}}^{-,1} - \boldsymbol{\mathfrak{z}}^{-,2}\|_{1, \alpha}(t)\right\}.
			\end{aligned}
		\end{equation}
		
		Concerning $A^{\pm}(t)$ we first apply the Chain Rule and the Mean Value Theorem in order to estimate the $\xCn{1}(\overline{\Omega}_3)$-norm of $G^{\pm,1}(\boldsymbol{\mathcal{Z}}^{\mp,1}(\cdot,0,t), t)-G^{\pm,1}(\boldsymbol{\mathcal{Z}}^{\mp,2}(\cdot,0,t), t)$, which results in
		\begin{equation}\label{equation:A+-1}
			\begin{aligned}
				A^{\pm}(t) & \leq &&  \sup\limits_{\mathbf{x} \in \Omega_3} \left|\boldsymbol{\mathcal{Z}}^{\mp,1}(\mathbf{x},0,t) - \boldsymbol{\mathcal{Z}}^{\mp,2}(\mathbf{x},0,t)\right| \|G^{\pm,1}\|_{1,\alpha,\Omega_3}(t)\\
				& && + \underset{:= A_1^{\pm}(t)}{\underbrace{\sum\limits_{l = 1,2} \sup\limits_{\mathbf{x} \in \Omega_3} \left[ \left|\boldsymbol{\nabla} G^{\pm,1}(\boldsymbol{\mathcal{Z}}^{\mp,1})\cdot\partial_l\boldsymbol{\mathcal{Z}}^{\mp,1} - \boldsymbol{\nabla} G^{\pm,1}(\boldsymbol{\mathcal{Z}}^{\mp,2})\cdot\partial_l\boldsymbol{\mathcal{Z}}^{\mp,2} \right| (\mathbf{x},0,t)  \right]}},
			\end{aligned}
		\end{equation}
		where $\left|G^{\pm,i}(\boldsymbol{\mathcal{Z}}^{\mp,i})\cdot\partial_l\boldsymbol{\mathcal{Z}}^{\mp,i}\right|(\mathbf{x},0,t)$ is short for $\left|G^{\pm,i}(\boldsymbol{\mathcal{Z}}^{\mp,i}(\mathbf{x},0,t),t)\cdot\partial_l\boldsymbol{\mathcal{Z}}^{\mp,i}(\mathbf{x},0,t)\right|$.
		Then, by taking into account that the flows $\boldsymbol{\mathcal{Z}}^{\pm}$ cannot leave $\Omega_3$, and using the Mean Value Theorem again, one has for $A_1^{\pm}(t)$ that
		\begin{equation}\label{equation:reg3}
			\begin{aligned}
				A_1^{\pm}(t) & \leq &&  \sum\limits_{l = 1,2} \sup\limits_{\mathbf{x} \in \Omega_3} \Big[ \Big|\boldsymbol{\nabla} G^{\pm,1}(\boldsymbol{\mathcal{Z}}^{\mp,1})\cdot\partial_l(\boldsymbol{\mathcal{Z}}^{\mp,1} - \boldsymbol{\mathcal{Z}}^{\mp,2} + \boldsymbol{\mathcal{Z}}^{\mp,2}) - \boldsymbol{\nabla} G^{\pm,1}(\boldsymbol{\mathcal{Z}}^{\mp,2})\cdot\partial_l\boldsymbol{\mathcal{Z}}^{\mp,2} \Big| (\mathbf{x},0,t)  \Big]\\
				& \leq && \sum\limits_{l = 1,2} \sup\limits_{\mathbf{x} \in \Omega_3} \Big[ \Big|\boldsymbol{\nabla} G^{\pm,1}(\boldsymbol{\mathcal{Z}}^{\mp,1})\cdot\partial_l(\boldsymbol{\mathcal{Z}}^{\mp,1} - \boldsymbol{\mathcal{Z}}^{\mp,2})\Big|(\mathbf{x},0,t) \\
				& && + \Big| (\boldsymbol{\nabla} G^{\pm,1}(\boldsymbol{\mathcal{Z}}^{\mp,1}) - \boldsymbol{\nabla} G^{\pm,1}(\boldsymbol{\mathcal{Z}}^{\mp,2}))\cdot\partial_l\boldsymbol{\mathcal{Z}}^{\mp,2} \Big| (\mathbf{x},0,t)  \Big]\\
				& \leq && C \|G^{\pm, 1}\|_{1,\alpha, \Omega_3}(t)  \sum\limits_{l = 1,2} \sup\limits_{\mathbf{x} \in \Omega_3} \Big| \partial_l\boldsymbol{\mathcal{Z}}^{\mp,1}(\mathbf{x},0,t) - \partial_l\boldsymbol{\mathcal{Z}}^{\mp,2}(\mathbf{x},0,t) \Big| \\
				& && + \|G^{\pm, 1}\|_{2,\alpha,\Omega_3} (t)  \sum\limits_{l = 1,2} \sup\limits_{\mathbf{x} \in \Omega_3} \Big|\partial_l\boldsymbol{\mathcal{Z}}^{\mp,2}(\mathbf{x},0,t) \Big| \Big|\boldsymbol{\mathcal{Z}}^{\mp,1}(\mathbf{x},0,t)-\boldsymbol{\mathcal{Z}}^{\mp,2}(\mathbf{x},0,t)\Big|,
			\end{aligned}
		\end{equation}
		noting that $\|G^{\pm, 1}\|_{2,\alpha,\Omega_3} (t) \leq \|G^{\pm, 1}\|_{\tilde{m}-1,\alpha,\Omega_3} (t)$ since $\tilde{m} \geq 3$. Hence,
		\begin{equation}\label{equation:A+-2}
			\begin{aligned}
				\omega_k(t) A_1^{\pm}(t)& \leq && K_4  \sum\limits_{l = 1,2} \sup\limits_{\mathbf{x} \in \Omega_3} \Big| \partial_l\boldsymbol{\mathcal{Z}}^{\mp,1}(\mathbf{x},0,t) - \partial_l\boldsymbol{\mathcal{Z}}^{\mp,2}(\mathbf{x},0,t) \Big| + K_4 \sup\limits_{\mathbf{x} \in \Omega_3} \Big|\boldsymbol{\mathcal{Z}}^{\mp,1}(\mathbf{x},0,t)-\boldsymbol{\mathcal{Z}}^{\mp,2}(\mathbf{x},0,t)\Big|,
			\end{aligned}
		\end{equation}
		where $K_4 = K_4(\nu, \mathbf{y}^*) > 0$ comes from estimating $\omega_k(t)\|G^{\pm, 1}\|_{\tilde{m}-1,\alpha,\Omega_3} (t)$ using the properties of $X_{\nu,k}$ as previously and by applying estimates similar to \eqref{equation:flowgradest} for including
		\[
		\sup\limits_{\mathbf{x} \in \Omega_3} \Big|\partial_l\boldsymbol{\mathcal{Z}}^{\mp,2}(\mathbf{x},0,t) \Big| \leq C 	\operatorname{exp}\left(\|\boldsymbol{\mathfrak{z}}^{\mp,2}\|_{1,\alpha,\Omega_3}(t)\right) \leq C \operatorname{exp}\left(\|\boldsymbol{\mathfrak{z}}^{\mp,2} - \mathbf{y}^*\|_{1,\alpha,\Omega_3}(t) + \|\mathbf{y}^*\|_{1,\alpha,\Omega_3}(t)\right),
		\]
		into $K_4$ as well. Therefore, by inserting \eqref{equation:A+-2} into \eqref{equation:A+-1}, with a constant $K_5 = K_5(\nu, \mathbf{y}^*) > 0$ which depends on the same quantities as $K_4$ and is independent of $(\overline{\boldsymbol{\mathfrak{z}}^{+,i}}, \overline{\boldsymbol{\mathfrak{z}}^{-,i}}) \in X_{\nu, k}$, $i = 1,2$, it holds that
		\begin{equation}\label{equation:A+-3}
			\begin{aligned}
				\omega_k(t)A^{\pm}(t) & \leq &&  K_5 \sup\limits_{\mathbf{x} \in \Omega_3} \left|\boldsymbol{\mathcal{Z}}^{\mp,1}(\mathbf{x},0,t) - \boldsymbol{\mathcal{Z}}^{\mp,2}(\mathbf{x},0,t)\right| + K_5  \sum\limits_{l = 1,2} \sup\limits_{\mathbf{x} \in \Omega_3} \Big| \partial_l\boldsymbol{\mathcal{Z}}^{\mp,1}(\mathbf{x},0,t) - \partial_l\boldsymbol{\mathcal{Z}}^{\mp,2}(\mathbf{x},0,t) \Big|.
			\end{aligned}
		\end{equation}
		In order to further estimate \eqref{equation:A+-3}, we first note that due to the definition of the flow maps $\boldsymbol{\mathcal{Z}}^{\pm}$ one has for $(\mathbf{x},t) \in \overline{\Omega}_3 \times [0,1]$ that
		\begin{equation}\label{equation:A+-4a}
			\begin{aligned}
				\begin{cases}
					\xdrv{}{t} \boldsymbol{\mathcal{Z}}^{\pm,1}(\mathbf{x},0,t) - \xdrv{}{t}  \boldsymbol{\mathcal{Z}}^{\pm,2}(\mathbf{x},0,t) & = \boldsymbol{\mathfrak{z}}^{\pm,1}(\boldsymbol{\mathcal{Z}}^{\pm,1}(\mathbf{x},0,t),t)  - \boldsymbol{\mathfrak{z}}^{\pm,2}(\boldsymbol{\mathcal{Z}}^{\pm,2}(\mathbf{x},0,t),t),\\
					\boldsymbol{\mathcal{Z}}^{\pm,1}(\mathbf{x},0,0) -  \boldsymbol{\mathcal{Z}}^{\pm,2}(\mathbf{x},0,0) & = \mathbf{0},
				\end{cases}
			\end{aligned}
		\end{equation}
		and for $l = 1,2$, $i = 1,2$ that
		\begin{equation}\label{equation:A+-4b}
			\begin{aligned}
				\begin{cases}
					\xdrv{}{t} \partial_l \mathcal{Z}^{\pm,1}_i(\mathbf{x},0,t) - \xdrv{}{t} \partial_l \mathcal{Z}^{\pm,2}_i(\mathbf{x},0,t) & =  \boldsymbol{\nabla} \mathfrak{z}^{\pm,1}_i(\boldsymbol{\mathcal{Z}}^{\pm,1}(\mathbf{x},0,t),t) \cdot \partial_l \boldsymbol{\mathcal{Z}}^{\pm,1}(\mathbf{x},0,t) \\
					& \quad - \boldsymbol{\nabla} \mathfrak{z}^{\pm,2}_i(\boldsymbol{\mathcal{Z}}^{\pm,2}(\mathbf{x},0,t),t)  \cdot \partial_l \boldsymbol{\mathcal{Z}}^{\pm,2}(\mathbf{x},0,t),\\
					\partial_l \boldsymbol{\mathcal{Z}}^{\pm,1}(\mathbf{x},0,0) -  \partial_l \boldsymbol{\mathcal{Z}}^{\pm,2}(\mathbf{x},0,0) & = \mathbf{0}.
				\end{cases}
			\end{aligned}
		\end{equation}
		Thus, integrating in \eqref{equation:A+-4a} - \eqref{equation:A+-4b} from $0$ to $t$, and using again the Mean Value Theorem as well as the properties of the respective flows, shall in the sequel allow to appropriately bound the following quantities:
		\begin{itemize}
			\item $W^{\pm}_0(\mathbf{x},t) := \left| \boldsymbol{\mathcal{Z}}^{\pm,1}(\mathbf{x},0,t) - \boldsymbol{\mathcal{Z}}^{\pm,2}(\mathbf{x},0,t) \right|$,\\
			\item $W^{\pm,l}_i(\mathbf{x},t) := \left| \partial_l \mathcal{Z}^{\pm,1}_i(\mathbf{x},0,t) - \partial_l \mathcal{Z}^{\pm,2}_i(\mathbf{x},0,t) \right|$, for $i,l \in \{1,2\}$.
		\end{itemize}
		Starting with $W^{\pm}_0$, and using the approach from \eqref{equation:B+-0} for $\boldsymbol{\mathfrak{z}}^{\pm,1} -\boldsymbol{\mathfrak{z}}^{\pm,2}$, one has the estimate
		\begin{equation*}\label{equation:A+-6}
			\begin{aligned}
				W^{\pm}_0(\mathbf{x},t) & \leq && \int_0^t \Big|\boldsymbol{\mathfrak{z}}^{\pm,1}(\boldsymbol{\mathcal{Z}}^{\pm,1}(\mathbf{x},0,s),s)  - \boldsymbol{\mathfrak{z}}^{\pm,2}(\boldsymbol{\mathcal{Z}}^{\pm,2}(\mathbf{x},0,s),s) \Big| \, ds\\
				& \leq && \int_0^t \Big|\boldsymbol{\mathfrak{z}}^{\pm,1}(\boldsymbol{\mathcal{Z}}^{\pm,1}(\mathbf{x},0,s),s)  - \boldsymbol{\mathfrak{z}}^{\pm,1}(\boldsymbol{\mathcal{Z}}^{\pm,2}(\mathbf{x},0,s),s) + \boldsymbol{\mathfrak{z}}^{\pm,1}(\boldsymbol{\mathcal{Z}}^{\pm,2}(\mathbf{x},0,s),s) -\boldsymbol{\mathfrak{z}}^{\pm,2}(\boldsymbol{\mathcal{Z}}^{\pm,2}(\mathbf{x},0,s),s) \Big| \, ds \\
				& \leq && \int_0^t \Big\{ \|\boldsymbol{\mathfrak{z}}^{\pm,1}\|_{1,\alpha,\Omega_3}(s) W^{\pm}_0(\mathbf{x},s) + \|\boldsymbol{\mathfrak{z}}^{\pm,1} -\boldsymbol{\mathfrak{z}}^{\pm,2}\|_{0,\alpha,\Omega_3}(s) \Big\} \, ds.
			\end{aligned}
		\end{equation*}
		Concerning the terms $W^{\pm,l}_i$, $i,l \in \{1,2\}$, we obtain from \eqref{equation:A+-4b} by integration and the reasoning of \eqref{equation:flowgradest}, that
		\begin{equation}\label{equation:A+-7}
			\begin{aligned}
				W^{\pm,l}_i(\mathbf{x},t) & \leq && \int_0^t \Big| \boldsymbol{\nabla}\mathfrak{z}^{\pm,1}_i (\boldsymbol{\mathcal{Z}}^{\pm,1}(\mathbf{x},0,s),s) \cdot \partial_l \boldsymbol{\mathcal{Z}}^{\pm,1}(\mathbf{x},0,s) - \boldsymbol{\nabla} \mathfrak{z}^{\pm,2}_i(\boldsymbol{\mathcal{Z}}^{\pm,2}(\mathbf{x},0,s),s) \cdot  \partial_l \boldsymbol{\mathcal{Z}}^{\pm,2}(\mathbf{x},0,s) \Big| \, ds\\
				& = && \int_0^t \Big| \boldsymbol{\nabla} \mathfrak{z}^{\pm,1}_i(\boldsymbol{\mathcal{Z}}^{\pm,1}(\mathbf{x},0,s),s) \cdot \left[ \partial_l \boldsymbol{\mathcal{Z}}^{\pm,1}(\mathbf{x},0,s) - \partial_l \boldsymbol{\mathcal{Z}}^{\pm,2}(\mathbf{x},0,s) \right] \\
				& && + \left[\boldsymbol{\nabla} \mathfrak{z}^{\pm,1}_i(\boldsymbol{\mathcal{Z}}^{\pm,1}(\mathbf{x},0,s),s) - \boldsymbol{\nabla} \mathfrak{z}^{\pm,2}_i(\boldsymbol{\mathcal{Z}}^{\pm,2}(\mathbf{x},0,s),s) \right] \cdot \partial_l \boldsymbol{\mathcal{Z}}^{\pm,2}(\mathbf{x},0,s) \Big| \, ds\\
				& \leq && \int_0^t \Bigg\{ \|\boldsymbol{\mathfrak{z}}^{\pm,1}\|_{1,\alpha,\Omega_3}(s) \left|\partial_l \boldsymbol{\mathcal{Z}}^{\pm,1}(\mathbf{x},0,s) - \partial_l \boldsymbol{\mathcal{Z}}^{\pm,2}(\mathbf{x},0,s)\right| \\
				& && + \left| \boldsymbol{\nabla} \mathfrak{z}^{\pm,1}_i(\boldsymbol{\mathcal{Z}}^{\pm,1}(\mathbf{x},0,s),s) - \boldsymbol{\nabla} \mathfrak{z}^{\pm,1}_i(\boldsymbol{\mathcal{Z}}^{\pm,2}(\mathbf{x},0,s),s) \right| \\
				& && + \left| \left[\boldsymbol{\nabla} \mathfrak{z}^{\pm,1}_i(\boldsymbol{\mathcal{Z}}^{\pm,2}(\mathbf{x},0,s),s)  -\boldsymbol{\nabla}\mathfrak{z}^{\pm,2}_i(\boldsymbol{\mathcal{Z}}^{\pm,2}(\mathbf{x},0,s),s) \right] \cdot \partial_l \boldsymbol{\mathcal{Z}}^{\pm,2}(\mathbf{x},0,s) \right| \Bigg\} \, ds\\
				& \leq && C(\nu, \mathbf{y}^*)\int_0^t \Big\{ \left[W_1^{\pm,l}(\mathbf{x}, s) + W_2^{\pm,l}(\mathbf{x}, s)\right] + \|\boldsymbol{\mathfrak{z}}^{\pm,1}\|_{2,\alpha,\Omega}(s) W_0^{\pm}(\mathbf{x},s) + \|\boldsymbol{\mathfrak{z}}^{\pm,1} - \boldsymbol{\mathfrak{z}}^{\pm,2}\|_{1,\alpha,\Omega_3}(s) \Big\} \, ds.\\
			\end{aligned}
		\end{equation}
		Above, we also used the idea of \eqref{equation:B+-0} in order to estimate
		\[
		\begin{cases}
			\|\boldsymbol{\nabla}\mathfrak{z}^{\pm,1}_i (\boldsymbol{\mathcal{Z}}^{\pm,1}(\mathbf{x},0,s),s)\|_{0,\alpha,\Omega_3} & \leq \|\boldsymbol{\mathfrak{z}}^{\pm,1}\|_{1,\alpha,\Omega_3}(s) \leq C(\nu, \mathbf{y}^*),\\
			\|\boldsymbol{\nabla} \mathfrak{z}^{\pm,1}_i(\boldsymbol{\mathcal{Z}}^{\pm,2}(\mathbf{x},0,s),s)  -\boldsymbol{\nabla}\mathfrak{z}^{\pm,2}_i(\boldsymbol{\mathcal{Z}}^{\pm,2}(\mathbf{x},0,s),s)\|_{0,\alpha,\Omega_3} & \leq \|\boldsymbol{\mathfrak{z}}^{\pm,1} - \boldsymbol{\mathfrak{z}}^{\pm,2}\|_{1,\alpha,\Omega_3}(s).
		\end{cases}
		\]
		Accordingly, there exists a constant $K_6 = K_6(\nu, \mathbf{y}^*)$, which does not depend on $\overline{\boldsymbol{\mathfrak{z}}^{+,i}}$ and $\overline{\boldsymbol{\mathfrak{z}}^{-,i}}$ for $i = 1, 2$, such that \eqref{equation:A+-7} implies for all $t \in [0,1]$ and $l \in \{1,2\}$ that
		\begin{equation*}\label{equation:A+-8}
			\begin{aligned}
				\sum\limits_{i = 0}^2 W^{\pm,l}_i(\mathbf{x},t) & \leq && K_6 \int_0^t \Big\{ \sum\limits_{i = 0}^2 W^{\pm,l}_i(\mathbf{x},s) + \|\boldsymbol{\mathfrak{z}}^{\pm,1} - \boldsymbol{\mathfrak{z}}^{\pm,2}\|_{1,\alpha,\Omega_3}(s) \Big\} \, ds,\\
			\end{aligned}
		\end{equation*}
		which allows the application of Gr\"onwall's inequality in order to obtain for $l = 1,2$ that
		\begin{equation}\label{equation:A+-9}
			\begin{aligned}
				\sum\limits_{i = 0}^2 W^{\pm,l}_i(\mathbf{x},t) & \leq && K_7(\nu, \mathbf{y}^*)\max\limits_{t \in [0,1]}\|\boldsymbol{\mathfrak{z}}^{\pm,1} - \boldsymbol{\mathfrak{z}}^{\pm,2}\|_{1, \alpha}(t).
			\end{aligned}
		\end{equation}
		As a result of \eqref{equation:A+-3} and \eqref{equation:A+-9} one gets for $A^{\pm}(t)$ with a constant $K_8 = K_8(K_5,K_7) > 0$, independent of $(\overline{\boldsymbol{\mathfrak{z}}^{+,i}}, \overline{\boldsymbol{\mathfrak{z}}^{-,i}}) \in X_{\nu, k}$, $i = 1,2$, that
		\begin{equation}\label{equation:A+-_final}
			\begin{aligned}
				\omega_k(t)A^{\pm}(t) & \leq &&  K_8\max\limits_{t \in [0,1]}\|\boldsymbol{\mathfrak{z}}^{\mp,1} - \boldsymbol{\mathfrak{z}}^{\mp,2}\|_{1, \alpha}(t).
			\end{aligned}
		\end{equation}
		
		In total, by combining \eqref{equation:A+-_final} and \eqref{equation:B+-1} in \eqref{equation:G+-_1_2_decomp}, one arrives for all $t \in [0,1]$ at
		\begin{equation}\label{equation:contr_estimate4}
			\begin{aligned}
				I_3^{\pm}(t) & = && \tilde{C}(\nu, \mathbf{y}^*) \int_{0}^{1} \left\| G^{\pm,1}(\boldsymbol{\mathcal{Z}}^{\mp,1}(\cdot,0,s), s)-G^{\pm,2}(\boldsymbol{\mathcal{Z}}^{\mp,2}(\cdot,0,s), s) \right\|_{0,\alpha,\Omega_3} \, ds \\
				& = && \tilde{C}(\nu, \mathbf{y}^*) \int_{0}^{1} \frac{\omega_k(s)}{\omega_k(s)} \left\| G^{\pm,1}(\boldsymbol{\mathcal{Z}}^{\mp,1}(\cdot,0,s), s)-G^{\pm,2}(\boldsymbol{\mathcal{Z}}^{\mp,2}(\cdot,0,s), s) \right\|_{0,\alpha,\Omega_3} \, ds \\
				& \leq && K_9 \left[\int_0^{1} \frac{1}{\omega_k(s)} \, ds \right] \max\limits_{t \in [0,1]} \left\{ \|\boldsymbol{\mathfrak{z}}^{+,1} -  \boldsymbol{\mathfrak{z}}^{+,2}\|_{1,\alpha, \Omega}(t) + \|\boldsymbol{\mathfrak{z}}^{-,1} -  \boldsymbol{\mathfrak{z}}^{-,2}\|_{1,\alpha, \Omega}(t)\right\} \\
				& \leq && \frac{\kappa}{6} \max\limits_{t \in [0,1]} \left\{ \|\boldsymbol{\mathfrak{z}}^{+,1} -  \boldsymbol{\mathfrak{z}}^{+,2}\|_{1,\alpha, \Omega}(t) + \|\boldsymbol{\mathfrak{z}}^{-,1} -  \boldsymbol{\mathfrak{z}}^{-,2}\|_{1,\alpha, \Omega}(t)\right\} ,
			\end{aligned}
		\end{equation}
		with $K_9 = K_9(\nu, \mathbf{y}^*)$ not depending on $(\overline{\boldsymbol{\mathfrak{z}}^+}, \overline{\boldsymbol{\mathfrak{z}}^-}) \in X_{\nu, k}$, which allows to increase $k$ and reduce $\|\mathbf{z}^{\pm}_0\|_{2,\alpha,\Omega}$ again, if needed, in order to have the last inequality in \eqref{equation:contr_estimate4}.
	\end{proof}	
	
	\begin{acknowledgement}
		{\bf Acknowledgments}: This research was partially supported by National Natural Science Foundation of China
		under Grant No. 11631008, National Key R\&D Program of China under Grant No. 2020YFA0712000, Strategic Priority Research Program of Chinese Academy of Sciences under Grant No. XDA25010402, and Shanghai Municipal Education Commission under Grant No. 2021-01-07-00-02-E00087.
		
	\end{acknowledgement}
	
	\bibliographystyle{plain}
	\bibliography{mhd2darx_v2}
\end{document}